\documentclass{amsart}
\pdfoutput=1

\usepackage{amsmath,amssymb, amsthm, amsfonts}
\usepackage{mathtools}
\usepackage[english]{babel}
\usepackage{microtype}

\usepackage{enumerate}
\usepackage{comment}
\usepackage[textsize=scriptsize]{todonotes}

\usepackage[numbers,sort,compress]{natbib}

\usepackage{xspace}

\usepackage{tikz,tikz-cd}
\usetikzlibrary{matrix,arrows,positioning,calc,decorations.markings}
\usepackage
[
colorlinks=true,
linkcolor=black,
anchorcolor=black,
citecolor=black,
urlcolor=black,
plainpages=false,
pdfpagelabels
]{hyperref}

\hypersetup{linktocpage}
\usepackage[capitalize]{cleveref}

\numberwithin{equation}{section}

\newtheorem{theorem}[equation]{Theorem}

\newtheorem{proposition}[equation]{Proposition}
\newtheorem{lemma}[equation]{Lemma} 

\newtheorem{conjecture}[equation]{Conjecture}

\theoremstyle{definition}
\newtheorem{example}[equation]{Example}
\newtheorem{remark}[equation]{Remark}

\newtheorem{definition}[equation]{Definition}
\newtheorem{naive}{Naive Persistent Whitehead Conjecture}

\newcommand\category[1]{\ensuremath{\mathbf{#1}}}
\DeclareMathOperator{\im}{im}

\DeclareMathOperator{\obj}{obj}

\DeclareMathOperator{\id}{Id}
\DeclareMathOperator{\Rips}{{\mathcal R}}
\newcommand{\Cech}{\mathcal{\check C}}
\DeclareMathOperator{\colim}{colim}
\DeclareMathOperator{\acolim}{fcolim}
\DeclareMathOperator{\hocolim}{hocolim}

\DeclareMathOperator{\Ner}{Ner}

\DeclareMathOperator{\fc}{\zeta}
\DeclareMathOperator{\Ex}{Ex}

\DeclareMathOperator{\Ho}{Ho}
\DeclareMathOperator{\ob}{ob}

\DeclareMathOperator{\comp}{\circ}
\DeclareMathOperator{\proj}{proj}

\DeclareMathOperator{\union}{\cup\ }

\newcommand{\bparagraph}[1]{\subsubsection*{#1}}
\newcommand{\SimpCat}[1]{\boldsymbol{[#1]}}

\newcommand{\kVect}[0]{\category{Vec}}

\newcommand{\Top}[0]{\category{Top}}

\newcommand{\crit}[0]{^\mathrm{crit}}
\newcommand{\Sb}[0]{\mathcal{S}}

\newcommand{\Int}[1]{\bar{#1}}

\newcommand{\Lan}[0]{L}
\newcommand{\ldf}[1]{\Lambda^{#1}}

\newcommand{\R}[0]{\mathbb R}

\newcommand{\B}[1]{\mathcal B\ifthenelse{\equal{#1}{}}{}{_{#1}}}
\newcommand{\Bopen}[1]{U\ifthenelse{\equal{#1}{}}{}{(#1)}}

\newcommand{\barc}[1]{\mathcal B\ifthenelse{\equal{#1}{}}{}{(#1)}}

\newcommand{\A}[0]{\mathbf I}
\newcommand{\Bb}[0]{\mathbf J}
\newcommand{\C}[0]{\mathbf C}
\newcommand{\D}[0]{\mathbf D}

\newcommand{\F}[0]{X}

\newcommand{\I}[0]{\mathbf I}
\newcommand{\J}[0]{\mathbf J}

\newcommand{\maptocolim}[2]{\mu^{#1}_{#2}}

\newcommand{\K}[0]{\mathbf K}

\newcommand{\Q}[0]{\mathbf Q}

\newcommand{\Z}[0]{\mathbb Z}
\newcommand{\ZCat}[0]{\mathbb Z}

\newcommand{\RCat}[0]{\mathbb{R}}

\newcommand{\CWCat}[0]{\category{CW}}

\newcommand{\Funs}[0]{\category{F}}
\newcommand{\Set}[0]{\category{Set}}
\newcommand{\Grp}[0]{\category{Grp}}

\newcommand{\pfd}{p.f.d.\@\xspace}

\newcommand{\htp}{\simeq}
\newcommand{\Map}{\textrm{Map}}
\newcommand{\Coh}{\textrm{Coh}}

\newcommand{\ab}{\mathbf{a}}
\newcommand{\bb}{\mathbf{b}}
\newcommand{\cb}{\mathbf{c}}
\newcommand{\db}{\mathbf{d}}
\newcommand{\eb}{\mathbf{e}}
\newcommand{\fb}{\mathbf{f}}
\newcommand{\gb}{\mathbf{g}}
\newcommand{\hb}{\mathbf{h}}

\begin{document}

\title[Universality of the Homotopy Interleaving Distance]{Universality of the Homotopy Interleaving Distance}

\author{Andrew J. Blumberg}
\address{Irving Institute for Cancer Dynamics, Columbia University, USA}
\email{andrew.blumberg@columbia.edu}
\author{Michael Lesnick}
\address{Department of Mathematics and Statistics, University at Albany -- SUNY, USA}
\email{mlesnick@albany.edu}

\date{}
\begin{abstract}
As a step towards establishing homotopy-theoretic foundations for topological data analysis (TDA), we introduce and study \emph{homotopy interleavings}
between filtered topological spaces.  These are homotopy-invariant analogues of interleavings, objects commonly used in TDA to articulate stability and inference theorems.  Intuitively, whereas a strict interleaving between filtered spaces $X$ and $Y$ certifies that $X$ and $Y$ are approximately isomorphic, a homotopy interleaving between $X$ and $Y$ certifies that $X$ and $Y$ are approximately weakly equivalent. 

The main results of this paper are that homotopy interleavings induce an extended pseudometric $d_{HI}$ on
filtered spaces, and that this is the universal
pseudometric satisfying natural stability and homotopy invariance
axioms.  To motivate these axioms, we also observe that $d_{HI}$ (or more generally, any pseudometric
satisfying these two axioms and an additional ``homology bounding" axiom) can be used to formulate lifts of several 
fundamental TDA theorems from the algebraic
(homological) level to the level of filtered spaces.  

Finally, we consider the problem of establishing a persistent Whitehead theorem in terms of homotopy interleavings.  We provide a counterexample to a naive formulation of the result.
\end{abstract}
\maketitle
\thispagestyle{empty}

\tableofcontents
 

\section{Introduction}
\label{Sec:Intro}
Topological data analysis (TDA) is a branch of statistics whose goal is to apply topology to analyze the global, non-linear, geometric features of data.  At a high level, the basic TDA workflow is easy to describe: Given a data set $P$, e.g., a finite metric space, we construct a diagram of topological spaces $F(P)$ whose topological structure encodes information about the shape of $P$, and then study $F(P)$ using familiar tools from algebraic topology.  The best known example of this workflow is \emph{persistent homology}, which takes $F(P)$ to be a filtered topological space and then applies homology with coefficients in a field to obtain simple, readily computed invariants of $P$ called \emph{barcodes}.  

Diagrams of topological spaces have been studied extensively in algebraic topology.  The essential mathematical difference between the TDA theory and classical topology is that in TDA, we usually work with \emph{metrics} between diagrams of spaces and their invariants.  We are typically more interested in whether a pair of objects are close in some suitably chosen metric (i.e., ``approximately equal") than in whether they are exactly equal.

From this vantage point, the well-known stability theory for persistent homology and its applications to TDA can be seen as a rudimentary form of ``approximate algebraic topology" for filtered spaces, in which a central role is played by such metrics.  In analogy with classical algebraic topology, one imagines that there should exist an ``approximate homotopy theory" for filtered spaces which serves as the foundation for this approximate algebraic topology; indeed, the beginnings of such a theory are already implicit in the proofs of well-known TDA results.  

The aim of this paper is to develop the formal language needed to start fleshing out this approximate homotopy theory.  
The essential first problem is to select a suitable notion of \emph{approximate weak equivalence} of filtered spaces; we focus primarily on this.  
Our results establish that our \emph{homotopy interleavings} and the metric they induce, the \emph{homotopy interleaving distance}, provide a notion of approximate weak equivalence that is very well behaved, relative to the needs of TDA.  
To obtain our results, we study how the metrics commonly used in TDA interact with the homotopy theory of diagrams of topological spaces.

Our first and main application of homotopy interleavings is to formulate theorems about persistent homology directly on the level of filtered spaces, rather than on the level of barcodes.  To explain, typical stability, inference, and approximation results for persistent homology assert that certain pairs of filtered spaces have barcodes which are similar.  In many cases, one can use homotopy interleavings to strengthen such results to ones formulated directly on the level of filtered spaces.  We consider several examples in this paper.  Additional applications of homotopy interleavings in a similar spirit are given in our recent work \cite{blumberg2020stability} on the stability of density-sensitive bifiltrations constructed from metric data.  

As another application, we study the problem of obtaining a persistent Whitehead theorem using homotopy interleavings.  We provide a counterexample to naive formulations of the result, and guided by this, present a weaker statement as a conjecture.  Following the release of the first version of this paper, Lanari and Scoccola \cite{lanari2020rectification} have proven this conjecture under a mild additional assumption.

In the remainder of the introduction, we state our main results and discuss applications of homotopy interleavings in detail. 
  
\subsection{Persistent Homology}\label{Sec:Persistent_Homology}
As mentioned above, persistent homology provides invariants of data
called \emph{barcodes}.  In the last two decades, these invariants have been studied extensively and 
 applied widely to scientific data 
\cite{edelsbrunner2010computational,carlsson2014topological}.  

Formally, a barcode is a multiset of intervals in a totally ordered set $A$; an interval in $A$ is a non-empty subset $I\subset A$ such that $s\in I$ whenever $r\leq s\leq t$ and $r,t\in I$.  In this paper, we'll be interested primarily in the cases $A=\R$ and  $A=[0,\infty)$.  Intuitively, each interval in the barcode
represents a topological feature of our data, and the
length of the interval is a measure of the robustness of that
feature to perturbations of the data.

For categories $\C$ and $\D$ with $\C$ small, let $\D^\C$ denote the category of functors $\C\to\D$ with morphisms
the natural transformations.  Let $\Top$ denote the category of compactly-generated
weakly Hausdorff (CGWH) topological spaces.  It is standard in homotopy theory to restrict attention to CGWH spaces; see Section~\ref{Sec:CGWH}.    Let $\mathbf{Simp}$ denote the category of abstract simplicial complexes.  Via geometric realization, we sometimes regard $\mathbf{Simp}$ as a subcategory of $\Top$.  
 
Regarding
$A$, together with its total order, as a category in the usual
way, we define a \emph{$A$-space} to be an object $\F$ of
$\Top^A$.  If for each $r\leq s\in A$, the internal map $\F_{r,s}\colon \F_r\to \F_s$ is an \emph{inclusion} (i.e., a homeomorphism onto its image), then we call $\F$ a \emph{filtration}.  In all of the examples of filtrations we consider in this paper, we will have $F_{r}\subset F_{s}$, with $F_{r,s}$ the usual subset inclusion.  One standard example in TDA is the following.
\begin{definition}
Given a finite metric space $P$, denote its metric as $d_P$.  We define a filtration $\Rips(P):\R\to \mathbf{Simp}$, the \emph{(Vietoris-)Rips filtration of $P$},  as follows:  For
$r\geq 0$, $\Rips(P)_r$ is the clique complex on the graph with vertex set $P$ and edge set $\{[p, q] \mid d_P(p,q)\leq 2r\}$; for $r<0$, 
$\Rips(P)_r:=\emptyset$.
\end{definition}

Let $\kVect$ denote the category of $k$-vector spaces over some fixed
field $k$.  A \emph{persistence module} is an object $M$ of $\kVect^A$.  We say $M$ is \emph{pointwise finite dimensional (\pfd)} if $\dim M_r<\infty$ for all $r\in A$.  For $I\subset A$ an interval, define the \emph{interval module} $k^I$ to be the persistence module given by 
\begin{align*}
k^I_r&=
\begin{cases}
k &{\textup{if }} r\in I, \\
0 &{\textup{ otherwise},}
\end{cases}
& k^I_{r,s}=
\begin{cases}
\id_k &{\textup{if }} r\leq s\in I,\\
0 &{\textup{ otherwise}.}
\end{cases}
\end{align*}
\begin{theorem}[Structure theorem for persistence modules~\cite{zomorodian2005computing, crawley2012decomposition,botnan2020decomposition}] For any \pfd persistence module $M$, there is a unique barcode $\B{M}$ such that \[M\cong \bigoplus_{I\in \B{M}} k^I.\]  
\end{theorem}

Letting $H_i\colon \Top\to \kVect$ denote the $i^{\mathrm{th}}$ homology
functor with coefficients in $k$, we obtain a barcode $\B{H_i\Rips(P)}$ for any finite metric space $P$ and $i\geq 0$; see \cref{Fig:RipsEx} for an illustration.  This construction of barcodes  in fact extends to compact metric spaces $P$, though in this case $H_i \Rips(P)$ needn't be \pfd, so a slightly different approach is required \cite{chazal2016observable,chazal2014persistence}.  By considering other kinds of filtrations, we can obtain other kinds of barcode invariants of data in the same way. 
\definecolor{aqaqaq}{rgb}{0.6274509803921569,0.6274509803921569,0.6274509803921569}
\begin{figure}
\centering
\begin{center}
\begin{tikzpicture}[scale=.75]
\coordinate (a) at (1,0);
\coordinate (b) at (0,1);
\coordinate (c) at (-1,0);
\coordinate (d) at (0,-1);
\coordinate (e) at (.707,.707);
\coordinate (f) at (-.707,.707);
\coordinate (g) at (-.707,-.707);
\coordinate (h) at (.707,-.707);
\coordinate (i) at  (1.2,.5);
\coordinate (j) at (.5,.3);

\draw [fill] (a) circle [radius=0.07];
\draw [fill] (b) circle [radius=0.07];
\draw [fill] (c) circle [radius=0.07];
\draw [fill] (d) circle [radius=0.07];
\draw [fill] (e) circle [radius=0.07];
\draw [fill] (f) circle [radius=0.07];
\draw [fill] (g) circle [radius=0.07];
\draw [fill] (h) circle [radius=0.07];
\draw [fill] (i) circle [radius=0.07];
\draw [fill] (j) circle [radius=0.07];
\end{tikzpicture}
\hskip20pt
\begin{tikzpicture}[scale=.75]
\coordinate (a) at (1,0);
\coordinate (b) at (0,1);
\coordinate (c) at (-1,0);
\coordinate (d) at (0,-1);
\coordinate (e) at (.707,.707);
\coordinate (f) at (-.707,.707);
\coordinate (g) at (-.707,-.707);
\coordinate (h) at (.707,-.707);
\coordinate (i) at  (1.2,.5);
\coordinate (j) at (.5,.3);

\draw[thick] (a) -- (i);
\draw[thick] (a) -- (j);
\draw[thick] (e) -- (i);
\draw[thick] (e) -- (j);

\draw [fill] (a) circle [radius=0.07];
\draw [fill] (b) circle [radius=0.07];
\draw [fill] (c) circle [radius=0.07];
\draw [fill] (d) circle [radius=0.07];
\draw [fill] (e) circle [radius=0.07];
\draw [fill] (f) circle [radius=0.07];
\draw [fill] (g) circle [radius=0.07];
\draw [fill] (h) circle [radius=0.07];
\draw [fill] (i) circle [radius=0.07];
\draw [fill] (j) circle [radius=0.07];
\end{tikzpicture}
\hskip20pt
\begin{tikzpicture}[scale=.75]
\coordinate (a) at (1,0);
\coordinate (b) at (0,1);
\coordinate (c) at (-1,0);
\coordinate (d) at (0,-1);
\coordinate (e) at (.707,.707);
\coordinate (f) at (-.707,.707);
\coordinate (g) at (-.707,-.707);
\coordinate (h) at (.707,-.707);
\coordinate (i) at  (1.2,.5);
\coordinate (j) at (.5,.3);

\fill[color=aqaqaq, fill=aqaqaq, fill opacity=0.4] (a) -- (e) -- (i) -- cycle;
\fill[color=aqaqaq, fill=aqaqaq, fill opacity=0.4] (a) -- (e) -- (j) -- cycle;
\fill[color=aqaqaq, fill=aqaqaq, fill opacity=0.4] (a) -- (i) -- (j) -- cycle;
\fill[color=aqaqaq, fill=aqaqaq, fill opacity=0.4] (e) -- (i) -- (j) -- cycle;

\draw[thick] (a) -- (i);
\draw[thick] (a) -- (j);
\draw[thick] (e) -- (i);
\draw[thick] (e) -- (j);
\draw[thick] (a) -- (e);
\draw[thick] (i) -- (j);
\draw[thick] (b) -- (e);
\draw[thick] (b) -- (f);
\draw[thick] (c) -- (f);
\draw[thick] (c) -- (g);
\draw[thick] (d) -- (g);
\draw[thick] (d) -- (h);
\draw[thick] (a) -- (h);
\draw[thick] (a) -- (i);
\draw[thick] (a) -- (j);
\draw [fill] (a) circle [radius=0.07];
\draw [fill] (b) circle [radius=0.07];
\draw [fill] (c) circle [radius=0.07];
\draw [fill] (d) circle [radius=0.07];
\draw [fill] (e) circle [radius=0.07];
\draw [fill] (f) circle [radius=0.07];
\draw [fill] (g) circle [radius=0.07];
\draw [fill] (h) circle [radius=0.07];
\draw [fill] (i) circle [radius=0.07];
\draw [fill] (j) circle [radius=0.07];
\end{tikzpicture}
\hskip20pt
\begin{tikzpicture}[scale=.75]
%
\coordinate (a) at (1,0);
\coordinate (b) at (0,1);
\coordinate (c) at (-1,0);
\coordinate (d) at (0,-1);
\coordinate (e) at (.707,.707);
\coordinate (f) at (-.707,.707);
\coordinate (g) at (-.707,-.707);
\coordinate (h) at (.707,-.707);
\coordinate (i) at  (1.2,.5);
\coordinate (j) at (.5,.3);

\fill[color=aqaqaq, fill=aqaqaq, fill opacity=0.4] (a) -- (e) -- (i) -- cycle;
\fill[color=aqaqaq, fill=aqaqaq, fill opacity=0.4] (a) -- (e) -- (j) -- cycle;
\fill[color=aqaqaq, fill=aqaqaq, fill opacity=0.4] (a) -- (i) -- (j) -- cycle;
\fill[color=aqaqaq, fill=aqaqaq, fill opacity=0.4] (e) -- (i) -- (j) -- cycle;

\fill[color=aqaqaq, fill=aqaqaq, fill opacity=0.4] (a) -- (b) -- (e) -- cycle;
\fill[color=aqaqaq, fill=aqaqaq, fill opacity=0.4] (a) -- (b) -- (i) -- cycle;
\fill[color=aqaqaq, fill=aqaqaq, fill opacity=0.4] (a) -- (b) -- (j) -- cycle;
\fill[color=aqaqaq, fill=aqaqaq, fill opacity=0.4] (a) -- (e) -- (h) -- cycle;

\fill[color=aqaqaq, fill=aqaqaq, fill opacity=0.4] (e) -- (b) -- (f) -- cycle;
\fill[color=aqaqaq, fill=aqaqaq, fill opacity=0.4] (b) -- (f) -- (c) -- cycle;
\fill[color=aqaqaq, fill=aqaqaq, fill opacity=0.4] (f) -- (c) -- (g) -- cycle;
\fill[color=aqaqaq, fill=aqaqaq, fill opacity=0.4] (c) -- (g) -- (d) -- cycle;
\fill[color=aqaqaq, fill=aqaqaq, fill opacity=0.4] (g) -- (d) -- (h) -- cycle;
\fill[color=aqaqaq, fill=aqaqaq, fill opacity=0.4] (d) -- (h) -- (a) -- cycle;

\fill[color=aqaqaq, fill=aqaqaq, fill opacity=0.4] (b) -- (e) -- (i) -- cycle;
\fill[color=aqaqaq, fill=aqaqaq, fill opacity=0.4] (b) -- (e) -- (j) -- cycle;
\fill[color=aqaqaq, fill=aqaqaq, fill opacity=0.4] (b) -- (j) -- (i) -- cycle;

\fill[color=aqaqaq, fill=aqaqaq, fill opacity=0.4] (h) -- (e) -- (i) -- cycle;
\fill[color=aqaqaq, fill=aqaqaq, fill opacity=0.4] (h) -- (e) -- (j) -- cycle;
\fill[color=aqaqaq, fill=aqaqaq, fill opacity=0.4] (h) -- (i) -- (j) -- cycle;

\fill[color=aqaqaq, fill=aqaqaq, fill opacity=0.4] (e)-- (f) -- (j) -- cycle;
\draw[thick] (a) -- (i);
\draw[thick] (a) -- (j);
\draw[thick] (e) -- (i);
\draw[thick] (e) -- (j);
\draw[thick] (a) -- (e);
\draw[thick] (i) -- (j);
\draw[thick] (b) -- (e);
\draw[thick] (b) -- (f);
\draw[thick] (c) -- (f);
\draw[thick] (c) -- (g);
\draw[thick] (d) -- (g);
\draw[thick] (d) -- (h);
\draw[thick] (a) -- (h);
\draw[thick] (a) -- (i);
\draw[thick] (a) -- (j);
\draw[thick] (a) -- (b);
\draw[thick] (b) -- (c);
\draw[thick] (c) -- (d);
\draw[thick] (d) -- (a);
\draw[thick] (e) -- (f);
\draw[thick] (f) -- (g);
\draw[thick] (g) -- (h);
\draw[thick] (h) -- (a);
\draw[thick] (e) -- (h);
\draw[thick] (h) -- (i);
\draw[thick] (h) -- (j);
\draw[thick] (i) -- (b);
\draw[thick] (j) -- (b);
\draw[thick] (f) -- (j);

\draw [fill] (a) circle [radius=0.07];
\draw [fill] (b) circle [radius=0.07];
\draw [fill] (c) circle [radius=0.07];
\draw [fill] (d) circle [radius=0.07];
\draw [fill] (e) circle [radius=0.07];
\draw [fill] (f) circle [radius=0.07];
\draw [fill] (g) circle [radius=0.07];
\draw [fill] (h) circle [radius=0.07];
\draw [fill] (i) circle [radius=0.07];
\draw [fill] (j) circle [radius=0.07];
\end{tikzpicture}
\hskip20pt
\begin{tikzpicture}[scale=.75]
%
\coordinate (a) at (1,0);
\coordinate (b) at (0,1);
\coordinate (c) at (-1,0);
\coordinate (d) at (0,-1);
\coordinate (e) at (.707,.707);
\coordinate (f) at (-.707,.707);
\coordinate (g) at (-.707,-.707);
\coordinate (h) at (.707,-.707);
\coordinate (i) at  (1.2,.5);
\coordinate (j) at (.5,.3);

\fill[color=aqaqaq, fill=aqaqaq, fill opacity=1] (i) -- (e) -- (b) -- (f) -- (c) -- (g) -- (d) -- (h) -- cycle;
\draw[thick] (a) -- (b);
\draw[thick] (a) -- (c);
\draw[thick] (a) -- (d);
\draw[thick] (a) -- (e);
\draw[thick] (a) -- (f);
\draw[thick] (a) -- (g);
\draw[thick] (a) -- (h);
\draw[thick] (a) -- (i);
\draw[thick] (a) -- (j);

\draw[thick] (b) -- (c);
\draw[thick] (b) -- (d);
\draw[thick] (b) -- (e);
\draw[thick] (b) -- (f);
\draw[thick] (b) -- (g);
\draw[thick] (b) -- (h);
\draw[thick] (b) -- (i);
\draw[thick] (b) -- (j);

\draw[thick] (c) -- (d);
\draw[thick] (c) -- (e);
\draw[thick] (c) -- (f);
\draw[thick] (c) -- (g);
\draw[thick] (c) -- (h);
\draw[thick] (c) -- (i);
\draw[thick] (c) -- (j);

\draw[thick] (d) -- (e);
\draw[thick] (d) -- (f);
\draw[thick] (d) -- (g);
\draw[thick] (d) -- (h);
\draw[thick] (d) -- (i);
\draw[thick] (d) -- (j);

\draw[thick] (e) -- (f);
\draw[thick] (e) -- (g);
\draw[thick] (e) -- (h);
\draw[thick] (e) -- (i);
\draw[thick] (e) -- (j);

\draw[thick] (f) -- (g);
\draw[thick] (f) -- (h);
\draw[thick] (f) -- (i);
\draw[thick] (f) -- (j);

\draw[thick] (g) -- (h);
\draw[thick] (g) -- (i);
\draw[thick] (g) -- (j);

\draw[thick] (h) -- (i);
\draw[thick] (h) -- (j);

\draw[thick] (i) -- (j);

\draw [fill] (a) circle [radius=0.07];
\draw [fill] (b) circle [radius=0.07];
\draw [fill] (c) circle [radius=0.07];
\draw [fill] (d) circle [radius=0.07];
\draw [fill] (e) circle [radius=0.07];
\draw [fill] (f) circle [radius=0.07];
\draw [fill] (g) circle [radius=0.07];
\draw [fill] (h) circle [radius=0.07];
\draw [fill] (i) circle [radius=0.07];
\draw [fill] (j) circle [radius=0.07];
\end{tikzpicture}
\end{center}
\ \\ \ \\
\includegraphics[scale=0.40]{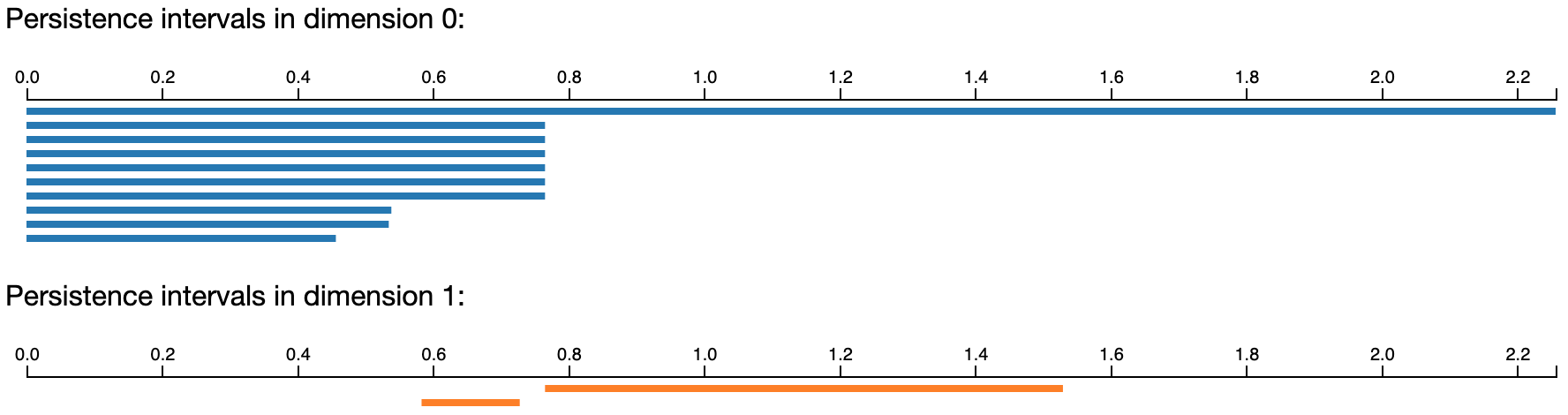} 
\caption{Rips complexes $\Rips(P)_r$ of a point cloud $P\subset \R^2$
for several choices of $r$, together with plots of the barcodes 
${\protect\B{H_0 \protect\Rips(P)}}$ and ${\protect\B{H_1 \protect\Rips(P)}}$.}
\label{Fig:RipsEx}.  
\end{figure}

\bparagraph{Bottleneck Distance}
As noted earlier, metrics on collections of topological invariants play a key role in TDA.  One standard choice of metric on barcodes is the \emph{bottleneck distance} $d_B$.  Roughly, for two barcodes $\mathcal C$ and $\mathcal D$, $d_B(\mathcal C,\mathcal D)$
is the maximum amount we need to perturb any interval in $\mathcal C$ to transform $\mathcal C$ into $\mathcal D$.  We now give the precise definition, restricting attention to barcodes consisting of intervals in $\R$.  

A \emph{matching} between sets $S$ and $T$ is a bijection between subsets of $S$ and $T$; this definition extends in the expected way to multisets.  

\begin{definition}[Bottleneck distance]\label{Def:Bottleneck}~
For $I\subseteq \R$ a nonempty interval and $\delta\geq 0$, define the interval
\[
\Ex(I,\delta):=\{r\in \R \mid \exists\, s\in I\,\textup{with} \,\,
|r-s|\leq \delta\}.
\]  	
For $\mathcal D$ a barcode and $\delta\geq 0$, let $\mathcal D^{\delta}$ denote the
collection of intervals in $\mathcal D$ that contain an interval of the form
$[r,r+\delta]$ for some $r\in \R$. 
A matching $\sigma$ between barcodes $\mathcal C$ and
$\mathcal D$ is called a \emph{$\delta$-matching} if 
\begin{enumerate}
\item $\sigma$ matches each interval in $\mathcal C^{2\delta}$ and $\mathcal D^{2\delta}$,
\item whenever $\sigma(I)=J$, we have $I\subseteq \Ex(J,\delta)$ and $J\subseteq \Ex(I,\delta)$.
\end{enumerate}
The \emph{bottleneck distance} $d_B$ is given by 
\[
d_B(\mathcal C,\mathcal D):=\inf\, \{\delta \mid
  \exists\textup{ a }\delta\textup{-matching between }\mathcal
  C\textup{ and }\mathcal D\}.
\]

\end{definition}

Recall that an extended
pseudometric on $S$ is a symmetric function
\[d\colon S\times S\to [0,\infty]\] 
satisfying the triangle inequality and $d(x,x)=0$ for all $x\in S$.
The bottleneck distance is an extended pseudometric
on barcodes, restricting to a genuine metric on the Rips barcodes of finite metric spaces.  
 In this paper, by a
\emph{distance} we will always mean an extended pseudometric.

\bparagraph{Stability of Persistent Homology of Metric Spaces}
The following fundamental stability result quantifies how persistent homology changes when the input data is perturbed.  Here, $d_{GH}$ denotes the Gromov-Hausdorff distance; see \cref{Def:GH_Dist}.

\begin{theorem}[Rips Stability \cite{chazal2009gromov,chazal2014persistence}]\label{Thm:RST}
For all finite metric spaces $P$ and $Q$,
\[d_B(\B{H_i\Rips(P)},\B{H_i\Rips(Q)})\leq d_{GH}(P,Q).\]
\end{theorem}

Using the appropriate definition of a barcode, this result generalizes to compact metric spaces \cite{chazal2014persistence}.  As we discuss in \cref{Sec:Interleavings}, one can give a yet more general, purely algebraic formulation of the stability of persistent homology, called the \emph{algebraic stability theorem}.

\subsection{Filtration-Level Refinement of the Rips Stability Theorem}\label{Sec:Intro_Homotopy_Interleavings}
A natural question is whether \cref{Thm:RST} can be regarded as
the consequence of some purely topological (homotopy-theoretic) result
about the filtrations $\Rips(P)$ and $\Rips(Q)$.  To obtain such a result, it suffices to identify a distance $d$ on $\R$-spaces satisfying the following two conditions:
\mbox{}
\begin{enumerate}[(i)]
\item For all 
metric spaces $P$ and
$Q$, \begin{equation}\label{Eq:First_Ineq} d(\Rips(P),\Rips(Q))\leq d_{GH}(P,Q). \end{equation}
\item {} [\emph{Homology bounding axiom}] For all $i\geq 0$ and $\R$-spaces $X,Y$ with $H_iX$ and $H_iY$ pointwise finite dimensional, 
\begin{equation*}
d_B(\B{H_iX}, \B{H_iY})\leq d(X,Y). \end{equation*}
\end{enumerate}
Clearly, the existence of such a distance $d$ implies~\cref{Thm:RST}.

\bparagraph{Axioms for a Distance on $\R$-Spaces}
Rather than work with directly with the inequality \eqref{Eq:First_Ineq} as we develop our theory, it turns out to be more more natural to introduce stability and homotopy invariance axioms for a distance $d$ on $\R$-spaces.  Together these axioms imply \eqref{Eq:First_Ineq}, but are more broadly applicable.

\begin{definition}
Given a topological space $T$ and a (not necessarily continuous) function $\gamma\colon T\to \R$, we define the \emph{sublevelset filtration}
$\Sb(\gamma)\colon \R\to \Top$ by taking \[\Sb(\gamma)_r=\gamma^{-1}(-\infty,r],\]
where each $\Sb_r$ is given the subspace
topology.
\end{definition}
Note that in the CGWH context, the subspace topology is understood to be that of~\cite[Definition 2.25]{strickland2009category}; this coincides with the standard subspace topology on subsets that are already CGWH in the standard topology.

For functions $\gamma,\kappa:T\to \R$, we let
\[
d_\infty(\gamma,\kappa)=\sup_{x\in T} |\gamma(x)-\kappa(x)|
\]

\begin{definition}\label{Def:Weakly_Equivalent}
For $\I$ any small category and functors $X,Y\colon \I\to \Top$, we say a natural transformation $f\colon X\to Y$ is an \emph{(objectwise) weak equivalence} if $f_a\colon X_a\to Y_a$ is a weak homotopy equivalence for all $a\in \ob \I$.  We let
\[
\begin{tikzcd}[ampersand replacement=\&,column sep=4ex]
X\ar["\simeq"]{r} \& Y
\end{tikzcd}
\]
denote an objectwise weak equivalence from $X$ to $Y$.  We say $X$ and $Y$ are \emph{weakly equivalent}, and write $X\htp Y$, if there exists a zigzag of objectwise weak equivalences
\[
\begin{tikzcd}[ampersand replacement=\&,column sep=2ex,row sep=2ex]
   \& W_1\ar["\simeq",swap]{dl}\ar["\simeq"]{dr}  \&           \& \cdots\ar["\simeq",swap]{dl}\ar["\simeq"]{dr}  \&                \&   W_n \ar["\simeq",swap]{dl}\ar["\simeq"]{dr}  \\
X \&                                                             \&  W_2 \&                                                                                                                       \& W_{n-1} \&                                                               \&Y.
\end{tikzcd}
\]
for some $n$.
\end{definition}
This is clearly an equivalence relation on objects, but it is
unwieldy.  As we explain in \cref{Sec:Model_Cats}, $X\simeq Y$ if and only if there exists a zig-zag of the following form:
\[
\begin{tikzcd}[ampersand replacement=\&,column sep=2ex,row sep=2ex]
\& W\ar["\simeq",swap]{dl}\ar["\simeq"]{dr}\\
X \& \& Y.
\end{tikzcd}
\]

\begin{definition}[Stability and homotopy invariance axioms]\label{Def:Two_Axioms}
We a say distance $d$ on $\R$-spaces is 
\begin{enumerate}
\item \emph{stable} if for any $T\in \ob \Top$ and functions $\gamma,\kappa\colon T\to \R$, 
\[d(\Sb(\gamma),\Sb(\kappa))\leq d_\infty(\gamma,\kappa),\]
\item \emph{homotopy invariant} if $d(X,Y)=0$ whenever $X\simeq Y$.
\end{enumerate}
\end{definition}
The following result is implicit in the original proof of the Rips stability theorem:
\begin{proposition}\label{Prop:Two_Axioms}
Any stable and homotopy invariant distance on $\R$-spaces satisfies \eqref{Eq:First_Ineq}, i.e., strengthens the Rips stability theorem to a filtration-level result.
\end{proposition}
We give a proof of \cref{Prop:Two_Axioms} in \cref{Rips_Stability_Refinement}, following a proof of the Rips stability theorem due to M\'emoli \cite{memoli2017distance}.  

\subsection{Properties of the Homotopy Interleaving Distance}\label{Sec:Properties}
The \emph{interleaving distance} $d_I$, the standard pseudometric on $\R$-spaces in the TDA literature, is stable and homology bounding, in the above senses, but is not homotopy invariant, and does not satisfy \eqref{Eq:First_Ineq}; see \cref{Rem:Not_Homotopy_Invariant}.  In \cref{Sec:HI}, we define \emph{homotopy interleavings} and the \emph{homotopy interleaving distance} $d_{HI}$ on $\R$-spaces by modifying the definition of $d_I$ to enforce the homotopy invariance axiom.
Our first main result is the following:

\begin{theorem}\label{Thm:Properties_of_dHI}
$d_{HI}$ is a distance on $\R$-spaces satisfying the stability, homotopy invariance, and homology bounding axioms.  
\end{theorem}

\cref{Prop:Two_Axioms} and \cref{Thm:Properties_of_dHI} together then tell us in particular that $d_{HI}$ satisfies \eqref{Eq:First_Ineq}.  Whereas it is trivial to show that $d_I$ satisfies the triangle inequality, our proof of the triangle inequality for $d_{HI}$ involves some work.  Our argument amounts to showing that some of the internal maps in a certain homotopy left Kan extension are weak homotopy equivalences.  Given the triangle inequality for $d_{HI}$ and the algebraic stability theorem, the rest of the proof of \cref{Thm:Properties_of_dHI} is trivial.

There are several pseudometrics on $\R$-spaces, besides $d_{HI}$, that satisfy the stability, homotopy invariance, and homology bounding axioms, and comparing these different choices can be difficult; see \cref{Sec:Related_Work} and \cref{Sec:Incoherent_Interleavings_And_Rectification}.  This raises the question of whether one can make a canonical choice of such a distance.  
The second main result of this paper, a simple axiomatic characterization of $d_{HI}$, provides an affirmative answer to this question: 
\begin{theorem}[Universality]\label{Thm:Universality}
If $d$ is any stable and homotopy invariant distance on $\R$-spaces, then $d\leq d_{HI}$.
\end{theorem}

\cref{Thm:Universality} is a homotopy-theoretic analogue of a universality result
for the interleaving distance on multiparameter persistence
modules over prime fields, established in \cite{lesnick2015theory}.  

We give a very brief outline of the proof of \cref{Thm:Universality}, deferring the details to \cref{Sec:Universality}: In analogy with the result of \cite{lesnick2015theory}, the proof of \cref{Thm:Universality} hinges on a lifting result, \cref{Prop:LiftToFunctions}, which says that for any two $\R$-spaces $X$ and $Y$ and $\delta>d_{HI}(X,Y)$, there exists a topological space $T$ and functions $\gamma^X,\gamma^Y:T\to \R$ such that $\Sb(\gamma^X)\htp X$, $\Sb(\gamma^Y)\htp Y$, and $d_\infty(\gamma^X,\gamma^Y)\leq \delta$.  To show this, the main technical step is to observe that a projectively cofibrant diagram of spaces indexed by a directed set $I$ is isomorphic to the sublevel filtration of an $I$-valued function; the desired lifting is then obtained taking a cofibrant replacement of a homotopy interleaving.  We review the definitions from homotopy theory needed to make sense of this in \cref{Sec:Homotopy_Theory_Prelims}.

\subsection{Extensions}
The definition of $d_{HI}$ extends immediately to the multiparameter persistence setting, i.e., to diagrams of spaces indexed by the product poset $\R^n$, as do both~\cref{Thm:Properties_of_dHI,Thm:Universality}.  In fact, further extensions are possible.  For simplicity's sake, we will restrict attention in this paper to the 1-parameter case, i.e., to $\R$-spaces, but we briefly discuss such extensions now.

Several works have considered general notions of interleavings \cite{bubenik2015metrics,scoccola2020locally,de2017theory}, and our main results extend to at least some of these settings.  For instance, Bubenik, de Silva, and Scott define an interleaving distance on diagrams of spaces indexed by a preordered set $P$ \cite[Section 2.5]{bubenik2015metrics}.  The definition depends on a choice of \emph{superlinear family} of morphisms of $P$.  Our definition of $d_{HI}$ extends to this setting, and when $P$ is a directed set, our universality result also extends readily.  The triangle inequality holds in this setting as well; our proof does not immediately extend, but a subsequent proof by Lanari and Scoccola \cite{lanari2020rectification} does extend readily.  The proof in \cite{lanari2020rectification} uses pullbacks, whereas ours uses pushouts; the pullback approach turns out to be better behaved for interleavings defined via superlinear families, and is also simpler.  Further extensions of our universality theorem, e.g., to simplicial sets, appear in the initial version of Lanari and Scoccola's paper on the arXiv, though these were cut from the final version of the paper for brevity's sake. 

Building on our work, Scoccola's Ph.D. thesis \cite{scoccola2020locally} gives a very general and thorough treatment of homotopy interleaving distances and the triangle inequality, as part of a broader program of reframing the theory of interleavings in the language of enriched categories.  
Among other contributions, the thesis introduces several novel examples and applications of generalized homotopy interleavings  and shows that under reasonable conditions, generalized homotopy interleaving distances are metrically complete, i.e., Cauchy sequences converge.

\subsection{Filtration-Level Formulations of Other TDA Results}\label{Sec:Other_Strengthenings}
Using any distance on $\R$-spaces satisfying our stability and homotopy invariance axioms, we can give space-level formulations of several other fundamental TDA results, besides the Rips stability theorem.  We briefly explain this now, deferring the details to  \cref{Sec:Applications}.

First, in \cite[Proposition 4.2]{chazal2014persistence}, Chazal et al. showed that if $X$ and $Y$ are simplicial filtrations whose vertex sets are related by a correspondence that is compatible with the filtrations in a suitable sense, then the barcodes of $X$ and $Y$ are close in the bottleneck distance.  As corollaries, the authors obtained stability results for several simplicial filtrations built from point cloud data, namely the Rips, \v Cech, Dowker, and witness filtrations; in particular, \cref{Thm:RST} is a corollary.
In \cref{Sec:Stabilty_Correspondences}, we observe that any stable and homotopy invariant distance $d$ satisfies a filtration-level analogue of \cite[Proposition 4.2]{chazal2014persistence}; the proof is essentially the same as our proof of \cref{Prop:Two_Axioms}.  Hence, each of the several stability results for simplicial filtrations given in \cite{chazal2014persistence} admits a filtration-level formulation in terms of $d$.

Second, an influential 2013 paper by Sheehy \cite{sheehy2013linear} introduced, for finite metric spaces $P$ of constant doubling dimension, a simplicial filtration whose size is linear in $|P|$ and whose persistent homology is a provably good approximation of that of $\Rips(P)$.  Subsequent work by Cavanna, Jahanseir, and Sheehy \cite{cavanna2015geometric} gave a more intuitive geometric variant of this construction, which also provides sparse approximation guarantees for \v Cech complexes.  Using any distance satisfying our axioms, the main approximation results of both papers lift to the level of filtrations.   In \cref{Sec:Sparse_Approximation}, we show this explicitly for the result of \cite{cavanna2015geometric}.

Third, we show in \cref{Prop:WLLN} that any distance satisfying our axioms can be used to formulate a simple weak law of large numbers for the persistent homology of \v Cech filtrations.  This follows readily from standard arguments, as, e.g,  in~\cite{chazal2013clustering,niyogi2009finding}.  
Notably, this weak law concerns the relationship between simplicial and non-simplicial filtrations, and thus illustrates the utility of a formalism which can handle non-simplicial filtrations.  Along very similar lines, one can give a filtration-level strengthening of a consistency result for a \v Cech-based estimator of the persistent homology of a probability density function, as given in \cite{chazal2013clustering} and \cite[Theorem 4.5.2]{lesnick2012multidimensional}.  Since the underlying ideas are quite similar to those of \cref{Prop:WLLN}, we will not discuss the details here.  

As another application of our ideas, our recent work \cite{blumberg2020stability} gives stability results and a weak law of large numbers for 2-parameter persistent homology, using a 2-parameter analogue of the homotopy interleaving distance.  The arguments depend only on 2-parameter analogues of our stability and homotopy invariance axioms.  

There are many other TDA theorems that bound the bottleneck distance between the barcodes of filtrations, besides those mentioned above, e.g., those in \cite{choudhary2021improved,brehm2018sparips,sheehy2021sparse,botnan2015approximating,brun2019sparse,dey2019simba,bobrowski2017topological,chowdhury2018functorial}.  We expect that nearly all of them can be formulated on the space level using any distance satisfying our axioms.

\subsection{Non-Universality of the Homotopy Commutative Interleaving Distance}\label{Sec:Non-Universality}
There is another simple definition of a stable and homotopy invariant interleaving distance on $\R$-spaces, the \emph{homotopy commutative interleaving distance} $d_{HC}$, given in terms of diagrams in the homotopy category of spaces; see \cref{Sec:Incoherent_Interleavings_And_Rectification}.  This is the first definition of a homotopy invariant interleaving distance that most TDA specialists would think of.  We conjecture that $d_{HC}\ne d_{HI}$, and hence that $d_{HC}$ is not universal.  In \cref{Ex:Unrectifiable}, we give an example of a based homotopy commutative interleaving diagram that cannot be rectified; this serves as evidence in support of the conjecture.

Following the release of the first version of this paper, Lanari and Scoccola \cite{lanari2020rectification} have proven this conjecture, via a similar example.  They also have shown that, conversely, $d_{HI}\leq 2\, d_{HC}$.  The latter inequality holds only in the 1-parameter setting; they have proven that in the multiparameter setting, $d_{HI}\not\leq c\, d_{HC}$ for any positive constant $c$. 

\subsection{Persistent Whitehead Conjectures}\label{Sec:Persistent Whitehead}
With a good definition of ``approximate weak equivalence" of
$\R$-spaces in hand, we are led to ask how other aspects of homotopy
theory might extend to the approximate setting.  For example, one has a Whitehead theorem for $\R$-spaces, which says that an objectwise weak equivalence of cofibrant $\R$-spaces is a homotopy equivalence.  It is natural to ask whether one can use the language of interleavings to formulate a persistent analogue of this result.  We explore this problem in \cref{Sec:Persistent_Whitehead}, considering along the way the question of how to define persistent homotopy groups.  We present an example showing that in its most naive formulation, the
persistent Whitehead theorem does not hold, even up to a constant, and offer a persistent Whitehead conjecture for cofibrant diagrams of CW-complexes of bounded dimension.  As mentioned above, Lanari and Scoccola \cite{lanari2020rectification} have proven this conjecture under a mild additional assumption.

\subsection{Other Related Work}\label{Sec:Related_Work}
This work is, in part, an outgrowth of a chapter in the second author's Ph.D. thesis, which introduced a pseudometric $d_{WI}$ on filtrations satisfying the stability, homotopy invariance, and homology bounding axioms considered in this paper \cite[Chapter 3]{lesnick2012multidimensional}.  This chapter showed that the strict interleaving distance on $\R$-spaces satisfies a universality property, and raised but did not answer the question of whether $d_{WI}$ is universal.  It is clear that $d_{WI}\leq d_{HI}$, but the problem of determining whether $d_{HI}=d_{WI}$ is non-trivial, due to technical issues related to homotopy coherence; see \cref{Sec:Incoherent_Interleavings_And_Rectification} and \cref{Remark:Interleaving_Vs_Approximate_Homotopy}.

Around the same time \cite{lesnick2012multidimensional} was completed, M\'emoli \cite{facundo2012banff,facundo2012atmcs,memoli2017distance} introduced a definition of a pseudometric $d_F$ on simplicial filtrations which can be used to provide a refinement of the Rips Stability theorem analogous to the one we give using $d_{HI}$. A different definition of $d_F$ is also implicit in the work of Chazal et al. \cite[Section 4.1]{chazal2014persistence}.  The definition of $d_F$ does not extend naively to non-simplicial filtrations, and it was shown in \cite[Section 6.8]{scoccola2020locally} that $d_F$ is not homotopy invariant.  However, our strategy for proving that $d_{HI}$ satisfies the triangle inequality was inspired by M\'emoli's proof of the triangle inequality for $d_F$ \cite{memoli2017distance}, which hinges on a pullback construction.

Besides the above, several other works have considered foundational aspects of persistent homotopy theory: Letscher introduced and studied a definition of persistent homotopy groups in the context of knot theory \cite{letscher2012persistent}, which was later also studied by Batan, Pamuk, and Varli \cite{batan2019persistent}.  As a step towards applying ideas from quantitative homotopy theory to TDA, Blumberg and Mandell~\cite{BM} introduced and studied filtered simplicial analogues of mapping spaces called \emph{contiguity complexes}.  M\'emoli and Zhao \cite{memoli2019persistent} studied several notions  of persistent homotopy groups of metric spaces.  
Frosini, Landi, and M\'emoli \cite{frosini2017persistent} strengthened the well known stability result for the persistent homology of $\R$-valued functions \cite{cohen2007stability} to pairs of functions with homotopy equivalent (but not necessarily homeomorphic) domains.  Jardine \cite{jardine2020persistent} introduced and studied \emph{controlled systems}, i.e., morphisms of $[0,\infty)$-spaces whose induced maps on persistent homotopy groups have kernels and cokernels with bounded persistence; such morphisms are closely related to interleavings \cite{bauer2015induced}.  Finally, building on ideas of Patel \cite{patel2018generalized}, Ghrist and Henselman-Petrusek \cite{ghrist2021saecular} recently introduced a novel generalization of the usual interval decomposition of a persistence module, which applies in particular to functors from a totally ordered set to the category of groups; this provides a notion of barcode for persistent homotopy groups.

\subsection{Outline of the Paper}

\cref{Sec:Homotopy_Theory_Prelims} provides a brief review of the 
tools from category theory and homotopy theory needed in our proofs.  \cref{Sec:Homotopy_Interleavings} reviews the
ordinary interleaving distance and introduces the
homotopy interleaving distance $d_{HI}$.  \cref{Sec:Triangle_Ineq}
gives the proof of the triangle inequality for $d_{HI}$, thereby
establishing \cref{Thm:Properties_of_dHI}, and \cref{Sec:Universality}
gives the proof of \cref{Thm:Universality}, our universality result for $d_{HI}$.   \cref{Rips_Stability_Refinement} gives the proof of
\cref{Prop:Two_Axioms}, which uses $d_{HI}$  to
lift the Rips stability theorem to the level of filtration.  \cref{Sec:Stabilty_Correspondences,Sec:Sparse_Approximation,Sec:WLLN} detail the other applications of $d_{HI}$ described in \cref{Sec:Other_Strengthenings}.  
\cref{Sec:Incoherent_Interleavings_And_Rectification} gives a
characterization of $d_{HI}$ in terms of
homotopy coherent diagrams of spaces, and explains the difficulties of
using homotopy commutative rather than homotopy coherent diagrams for
this.  \cref{Sec:Persistent_Whitehead} discusses the persistent Whitehead problem.

\subsection*{Acknowledgments}

We thank David Blanc, Gunnar Carlsson, Rick Jardine, Tyler Lawson, Mike Mandell, Facundo M\'emoli, Luis Scoccola, and Hiro Tanaka for helpful discussions.  We also thank Fabian Roll for pointing out an error in \cref{Sec:Hoco} of the first version of this paper, and an anonymous reviewer for very helpful suggestions which improved the exposition of the paper.  Both authors thank the Institute for
Mathematics and its Applications for its hospitality and support.  In addition, Lesnick thanks Robert Adler, Raul Rabadan, Jon Cohen, the Institute for Advanced Study, and the Princeton Neuroscience Institute for their support during various phases of this project. Lesnick was partially supported by NSF grant DMS-1128155, NIH grants U54-CA193313-01 and T32MH065214, funding from the IMA, and an award from the J. Insley Blair Pyne Fund.  Blumberg was partially supported by NIH grant 5U54CA193313 and AFOSR grant FA9550-15-1-0302.

\section{Background}\label{Sec:Homotopy_Theory_Prelims}
  In this section, we briefly review standard ideas from category theory and homotopy theory that we will need in the remainder of the paper: We discuss CGWH spaces, comma categories, Kan extensions, model categories, homotopy colimits, and homotopy Kan extensions.  We assume familiarity with other basic concepts from category theory, particularly limits and colimits.  For an introduction to category theory, see \cite{mac1998categories} or \cite{riehl2017category}.

\subsection{CGWH Spaces}\label{Sec:CGWH}
Recall that in \cref{Sec:Persistent_Homology} we defined $\Top$ to be the category of
compactly-generated weakly Hausdorff (CGWH) spaces.  It is
standard in modern homotopy theory to restrict attention to this
category because it contains most spaces that arise in practice and
its mapping spaces behave well \cite{may1999concise,riehl,strickland2009category}.  Specifically, the isomorphism of sets
$\Map(X \times Y, Z) \cong \Map(X, \Map(Y,Z))$ becomes a homeomorphism
in this category when we give the mapping spaces the compact-open
topology.  Any locally
compact Hausdorff space is a CGWH space, so in particular any space
homeomorphic to a CW complex 
complex 
is CGWH.  In this paper, working with CGWH spaces will be convenient for proving the university of $d_{HI}$.  

Recall that the category $\mathbf E$ is said to be \emph{cocomplete (respectively, complete)} if for any small category $\I$ and functor $F:\I\to \mathbf E$, the colimit (limit) of $F$ exists.  The category $\Top$ of CGWH spaces is
complete and cocomplete, although the colimits are not
always the same as those in the category of all topological spaces;
e.g., see~\cite{strickland2009category}.

Henceforth, all topological constructions in this paper will be 
carried out in the context of CGWH spaces.  We
will not comment on this point further, except when necessary in \cref{Sec:Concrete_Model}. 
\subsection{Comma Categories}\label{Sec:Slice}
Given a functor $G \colon \C \to \D$ and an object $d
\in \D$, the \emph{comma category} $G\downarrow d$ is the category whose objects are pairs $(c,\gamma\colon Gc\to d)$ where $c\in \ob \C$, and 
whose morphisms $\kappa:(c,\gamma) \to (c',\gamma')$ are morphisms $c\to c'$ in $\C$ such that the following diagram commutes:
\[
\begin{tikzcd}
  Gc \ar[rr,"G\kappa"] \ar[dr,"\gamma",swap] & & Gc' \ar[dl,"\gamma'"] \\
  &                d & &
\end{tikzcd}
\] 
Dually, $d \downarrow G$ is the category whose objects are pairs $(c,\gamma\colon d\to Gc)$ where $c\in \ob \C$, and 
whose morphisms $\kappa:(c,\gamma) \to (c',\gamma')$ are morphisms $c\to c'$ in $\C$ such that the following diagram commutes:
\[
\begin{tikzcd}
  Gc \ar[rr,"G\kappa"]  & & Gc' \\
  &                d\ar[ul,"\gamma",swap] \ar[ur,"\gamma'"] & &
\end{tikzcd}
\] 

\begin{remark}\label{rem:Comma_Category_Special_Case}
In our proof of the triangle inequality for $d_{HI}$, we will encounter the following simple kinds of comma categories: Recall that a category $\D$ is called \emph{thin} if for all $d,d'\in \ob \D$, there is at most one morphism from $d$ to $d'$.
Let $\C$ be a full subcategory of a thin category $\D$.  For $G:\C\to \D$ the inclusion and $d\in \ob \D$, we have that $G\downarrow d$ is (up to
canonical isomorphism) the full subcategory of $\C$ with object
set \[\{c\in \ob \C\mid \exists \textup{ a morphism $c\to d$ in
  $\D$}\}.\] The category $d\downarrow G$ has an analogous description.
\end{remark}

\subsection{Kan Extensions}\label{Sec:Kan}
A {\em left Kan extension} of a functor $F \colon \C \to \mathbf E$ along a functor $G \colon
\C \to \D$ is
a functor $L_G F \colon \D \to \mathbf E$
\[
\begin{tikzcd}
\C \ar[r,"F"] \ar[d,"G", swap] & \mathbf E \\
\D \ar[ur,"L_G F", swap] &
\end{tikzcd}
\]
together with a natural transformation
\[
\eta \colon F \Rightarrow L_G F \circ G
\]
that is universal in the sense
that for any other pair
\[
(H \colon \D \to \mathbf E, \gamma \colon F \Rightarrow H \circ G),
\]
$\gamma$ factors uniquely through $\eta$.  

A {\em right Kan
  extension} is defined by reversing the direction of the natural
transformations, i.e., it is a functor $R_G F \colon \D \to \mathbf E$ together with a
natural transformation $\eta \colon R_G \circ G \Rightarrow F$ satisfying the analogous universal property.  

In the case that $\C$ is small and $\mathbf E$ is cocomplete, the left Kan extension exists, and one has 
a pointwise formula for it in terms of colimits: For all $d\in \ob \D$, we have
\begin{equation}\label{Eq:LKE_Pointwise}
(L_{G} F)_d=\colim_{\,G\downarrow d} F
\end{equation}
and moreover, the internal morphisms in $L_G F$ are given by the
universality of the colimit \cite[Theorem X.3.1]{mac1998categories}.  Dually, for $\mathbf E$ complete, the right Kan extension exists, and one has a pointwise formula in terms of limits.  

For $\mathbf E$ cocomplete, the left Kan extension is functorial, i.e., the functors $L_GF$ assemble into a functor \[L_G:\mathbf E^\C\to \mathbf E^\D.\]
Dually, for $\mathbf E$ complete, the right Kan extension is functorial in the same sense.

\subsection{Model Categories}\label{Sec:Model_Cats}

The basic object of study in the homotopy theory of topological spaces
is the {\em homotopy category} $\Ho(\Top)$, obtained from $\Top$ by
formally inverting the weak homotopy
equivalences.  However, it turns out that many constructions, e.g., homotopy (co)limits, are difficult to carry
out directly in $\Ho(\Top)$.  Thus, we instead do these 
constructions in $\Top$ and study their interaction with weak
equivalences.  For this, additional scaffolding on $\Top$ is usually employed, in the form of
distinguished maps called {\em cofibrations} (which generalize closed
inclusions and are intended to have ``nice'' quotients) and {\em
  fibrations} (which generalize bundles and are intended to have
``nice'' fibers).  

A model category is an abstraction of this structure.  In what follows, we give the definition of a model category and discuss a few ideas from model category theory that will be needed in this paper.  For a detailed
introduction to model categories we recommend the survey
article~\cite{dwyer1995homotopy} or, for comprehensive treatments, the books~\cite{hovey,hirschhorn}.

\begin{definition}\label{Def:Model_Cat}
A \emph{model category} \cite[Definition 1.1.3]{hovey} is a complete and cocomplete category $\C$, together with three distinguished
collections of morphisms in $\C$, called the \emph{weak equivalences},
\emph{fibrations}, and \emph{cofibrations}, satisfying the four axioms below.  We say a (co)fibration is \emph{acyclic} if it is also a weak equivalence.

\begin{enumerate}

\item The weak equivalences contain all isomorphisms and
  satisfy the ``two out of three property'': for maps $f \colon X \to
  Y$ and $g \colon Y \to Z$, if any two of $f$, $g$, $g \circ f$ are
  weak equivalences then so is the third.

\item The weak equivalences, cofibrations, and fibrations are closed under retract; that is, if there is a commutative diagram
\[
\begin{tikzcd}
  X \ar[]{r} \ar["f",swap]{d} & Y\ar["g",swap]{d} \ar[]{r} &X\ar["f",swap]{d}\\
  X' \ar[]{r} & Y'\ar[]{r} &X',
\end{tikzcd}
\]
where the horizontal composites are the identity and $g$ is in the
class, then so is $f$.

\item In the commutative square
\[
\begin{tikzcd}
  A \ar[]{r} \ar["f",swap]{d} & X\ar["g"]{d} \\
  B \ar[]{r} & Y,
\end{tikzcd}
\]
if either
\begin{enumerate}
\item $f$ is an acyclic cofibration and $g$ is
  a fibration, or
\item $f$ is a cofibration and $g$ is an acyclic fibration,
\end{enumerate}
then there exists a lift $B \to X$ that makes the diagram commute.

\item Any morphism in $\C$ factors functorially as a composite of a
cofibration followed by an acyclic fibration and also as a composite of an
acyclic cofibration followed by a fibration.

\end{enumerate}

(It is sometimes convenient to drop the requirement that the
factorizations of axiom 4 are functorial, but we will
work in situations where this holds.)
\end{definition}

The following standard fact will be useful to us in our proof of the triangle inequality for $d_{HI}$.
\begin{proposition}[{\cite[Proposition 3.14]{dwyer1995homotopy}}]\label{Prop:Cobase_Change}
Given a pushout square in $\C$
\[
\begin{tikzcd}
  A \ar[]{r} \ar["f",swap]{d} & X\ar["g"]{d} \\
  B \ar[]{r} & Y,
\end{tikzcd}
\]
\begin{itemize}
\item[(i)] if $f$ is a cofibration, then so is $g$,
\item[(ii)] if $f$ is an acyclic cofibration, then so is $g$.
\end{itemize}
The dual result (involving pullbacks and (acyclic) fibrations) also holds.
\end{proposition}

Since a model category $\C$ is complete and cocomplete, it has an initial object $\emptyset$ and a final
object $\ast$.  We say $X\in \ob \C$ is \emph{cofibrant} if the unique
morphism $\emptyset\to X$ is a cofibration; dually, an object is
{\em fibrant} if the unique morphism $X \to \ast$ is a fibration.
Applying the functorial factorization axiom above to morphisms $\emptyset \to X$ 
yields a \emph{cofibrant replacement functor}
$Q\colon \C\to \C$ with each $QX$ cofibrant, and
a natural transformation $Q\to \mathrm{Id}_\C$ which is an acyclic
fibration on each object of $\C$.

In many cases, the cofibrant objects are inductively built from simple
pieces in a way that generalizes the notion of a CW complex.  In
particular, when a model category is \emph{compactly generated} (e.g., see 
\cite[Section 15.2]{mayponto}) we have a set of maps $I$, called {\em generating cofibrations}, such that cofibrations are built
from countable filtered colimits of ``cell attachments'' along maps in
$I$.  Here is a precise statement:

\begin{proposition}\label{prop:cell-complex}
In a compactly generated model category, an object is cofibrant if and only if it is a retract of
the colimit (union) of a diagram
\[
Z_0 \to Z_1 \to Z_2 \to \ldots,
\]
where $Z_0 = \emptyset$ and $Z_i$ is formed from $Z_{i-1}$ as a pushout
\[
\begin{tikzcd}
  \coprod_i X_i  \ar[]{r} \ar["\coprod_i I_i",swap]{d} & Z \ar[]{d} \\
  \coprod_i Y_i  \ar[]{r} & Z',
\end{tikzcd}
\]
with each $I_i$ a generating cofibration.
\end{proposition}

\begin{example}
In the \emph{standard model
  structure} on $\Top$, the weak equivalences are the weak homotopy equivalences and the fibrations are the Serre
fibrations; the cofibrations are then determined from the acyclic
fibrations by the model category axioms.  Henceforth, weak equivalences, fibrations,
and cofibrations of topological spaces will be understood to be those
in the standard model structure.  This model category is compactly generated, where the generating cofibrations are the boundary
inclusions $S^{n-1} \to D^{n}$; the cofibrant objects are retracts of
cell complexes, and the cofibrations in the
standard model structure admit a concrete description, as
\emph{retracts of generalized CW inclusions}~\cite{dwyer1995homotopy}. 
\end{example}

\begin{example}\label{Ex:Projective_Model}
For any small category $\I$, there exists a model category structure
on $\Top^\I$, the \emph{projective model structure}, for which the
weak equivalences are the objectwise weak equivalences and the
fibrations are the objectwise fibrations~\cite[Section 11.6]{hirschhorn}.  For most choices of $\I$,
objectwise cofibrations are not necessarily cofibrations. 
However, it is straightforward to check that if $X\in
\Top^\I$ is cofibrant, then each object in $X$ is cofibrant, and each
internal map in $X$ is a cofibration.  All objects are fibrant in the standard model structure on $\Top$,
hence all objects of $\Top^\I$ are fibrant in the projective model
structure.  

The projective model
structure on $\Top^\I$ is compactly generated, with generating cells
the {\em free diagrams} on the generating cofibrations in $\Top$ \cite[Theorem 6.5]{MMSS}.  Here the free diagram on a space $A$ at
$i \in \ob \I$, denoted $F_i A$, is specified by the formula $F_i A(j) =
\hom_{\I}(i,j) \times A$, and the free diagram at $i$ of a map $A\to B$ is the induced map $F_iA\to F_iB$. 
\end{example}

\bparagraph{Homotopy Categories}
One can construct the associated \emph{homotopy
category} $\Ho(\C)$ of any model category $\C$ \cite[Definition
  5.6]{dwyer1995homotopy}.  The category $\Ho(\C)$ has 
the same collection of objects as $\C$ and is equipped with a functor $\Pi^{\C} \colon \C\to
\Ho(\C)$ which is the identity on objects. $\Pi^{\C}$ is the 
\emph{localization} of $\C$ with respect to the weak equivalences \cite[Theorem 6.2]{dwyer1995homotopy}),
i.e., it maps weak equivalences to isomorphisms, and for any functor
$F\colon \C\to \D$ with this property, there exists a unique functor
$G\colon \Ho(\C)\to \D$ such that the following diagram commutes:
\[
\begin{tikzcd}
  \C \ar["F"]{r} \ar["\Pi^\C",swap]{d} & \D \\
  \Ho(\C) \ar["G",swap]{ur} &
\end{tikzcd}
\]
In particular, up to equivalence of categories, $\Ho(\C)$ depends only on the weak equivalences of $\C$, not on the (co)fibrations.  We let \[
\begin{tikzcd}[ampersand replacement=\&,column sep=4ex]
X\ar["\simeq"]{r} \& Y
\end{tikzcd}
\]
denote a weak equivalence from $X$ to $Y$.  We say $X$ and $Y$ are \emph{weakly equivalent} and write $X\htp Y$ if
$X,Y\in \ob \C$ are isomorphic in $\Ho(\C)$.  

The fact that $\Pi^{\C}$ is a localization implies that $X\htp Y$ if and only if there is a zigzag of weak equivalences in $\C$ connecting $X$ and $Y$.  Thus, two diagrams of spaces are weakly equivalent with respect to the projective model structure if and only if they are weakly equivalent in the sense of \cref{Def:Weakly_Equivalent}.

In fact, one can check that in any model category $\C$, $X\htp Y$ if and only if there exists a diagram of weak equivalences
\[
\begin{tikzcd}[,column sep=4ex,row sep=2ex]
& Z_1 & \ar["\simeq",swap]{l} Z_2 \ar["\simeq"]{dr} & \\
X \ar["\simeq"]{ur} & & & Y.
\end{tikzcd}
\]
See, e.g., \cite[Section 3.2]{dwyer2005homotopy}.  
Moreover, it is readily checked that if either all objects of $\C$ are fibrant or all objects are cofibrant, then 
$X\htp Y$ if and only if there exist weak equivalences
\[
\begin{tikzcd}[ampersand replacement=\&,column sep=2ex,row sep=2ex]
\& Z\ar["\simeq",swap]{dl}\ar["\simeq"]{dr}\\
X \& \& Y.
\end{tikzcd}
\]
In particular, this applies to the projective model structure, whose objects are all fibrant.

\subsection{Homotopy Colimits and Homotopy Left Kan extensions}\label{Sec:Hocolims_And_LKEs}
Homotopy (co)limits are analogues of (co)limits that are invariant
under weak equivalence.  Similarly, homotopy Kan extensions are analogues of Kan extensions
that are invariant under weak equivalence.  The proofs of our main results do not require explicit use of homotopy (co)limits or homotopy Kan extensions.  However, the construction underlying our proof of the triangle inequality for $d_{HI}$ in \cref{Sec:Triangle_Ineq} can be interpreted as a homotopy left Kan extension, and is arguably best viewed in that light.  Moreover, homotopy left Kan extensions of diagrams of spaces are given pointwise in terms of homotopy colimits, so understanding homotopy colimits is helpful for understanding homotopy left Kan extensions.  We therefore briefly discuss homotopy colimits and homotopy left Kan extensions here.

There are several (weakly
equivalent) ways to define homotopy colimits.  We will give a definition 
in terms of model categories and derived functors.  
Alternatively, one can define homotopy colimits using explicit formulas (e.g., the Bousfield-Kan formula), via 
a homotopy coherent analogue of
the universal property of colimits~\cite{vogt1973homotopy}, or using
the language of homotopical categories and homotopy initial
objects~\cite{dwyer2005homotopy}.  Thorough discussions of homotopy
colimits can be found in~\cite{riehl,shulman2006homotopy,dugger2008primer}.  

For model categories $\C$ and $\D$ and a functor $F\colon \C\to \D$,
the \emph{total left derived functor of $F$} is the functor $\ldf{F}\colon \Ho(\C)
\to \Ho(\D)$ given by the right Kan extension of $\Pi^\D F\colon \C\to
\Ho(\D)$ along $\Pi^\C\colon \C\to \Ho(\C)$.  We say a functor $\tilde F\colon \C\to \D$ \emph{computes $\ldf{F}$} if the following diagram commutes, up to natural isomorphism:
\[
\begin{tikzcd}
\C \ar{r}{\tilde F}\ar[swap]{d}{\Pi^\C} &\D\ar{d}{\Pi^\D}\\
\Ho(\C) \ar[swap]{r}{\ldf{F}} &\Ho(\D)
\end{tikzcd}
\]
If $F$ preserves weak equivalences between cofibrant objects, then the total left derived functor of $F$ exists, and for $Q\colon \C\to \C$ a
cofibrant replacement functor, $F\circ Q$ computes $\ldf{F}$ 
\cite[Section 9]{dwyer1995homotopy}.   
   
Let $\colim\colon\Top^\I\to \Top$ denote the colimit functor.  We define \[\hocolim\colon \Ho(\Top^\I)\to \Ho(\Top),\] the
\emph{homotopy colimit functor}, to be the total left derived functor
of $\colim$, with respect to the projective and standard model structures.
Since $\colim$ sends weak equivalences between cofibrant
diagrams in $\Top^\I$ to weak equivalences~\cite[Remark 9.8, Lemma 9.9
  and Proposition 10.7]{dwyer1995homotopy}, $\colim \comp Q$ computes $\hocolim$.

For a functor
$F\colon \I\to \J$, we define the \emph{homotopy left Kan
extension} to be the total left derived functor of $\Lan_F\colon \Top^\I\to \Top^\J$.
As with homotopy colimits, $\Lan_F \circ Q$ computes the homotopy left Kan
extension.  Moreover, the homotopy left Kan
extension can be given explicitly at each index as a homotopy colimit, via a formula analogous to \cref{Eq:LKE_Pointwise}; see, e.g.,  \cite[Proposition 1.14]{cisinski2009locally}.    

\section{The Homotopy Interleaving Distance}\label{Sec:Homotopy_Interleavings}

In this section, we define homotopy interleavings and the homotopy
interleaving distance.  We begin by recalling the definition of
ordinary interleavings. 
\subsection{Interleavings}\label{Sec:Interleavings}
Given a thin category $\C$, a functor $F\colon \C\to \D$, and a morphism $g\colon a\to b$ in $\C$, we denote $F(g)$ as $F_{a,b}$.  

For $\delta\geq 0$, let the \emph{$\delta$-interleaving category}, denoted $\I^\delta$, be the thin category with object set $\R\times \{0,1\}$ and a morphism $(r,i)\to (s,j)$ if and only if either  
\begin{enumerate}
\item $r+\delta \leq s$, or
\item $i=j$ and $r\leq s$;
\end{enumerate}
There are evident functors 
\[
E^0,E^1\colon\RCat\to \I^\delta
\]
mapping $r\in \R$ to $(r,0)$ and $(r,1)$, respectively.

\begin{definition}
For $\C$ any category and functors $X,Y\colon \RCat \to \C$, we define a \emph{$\delta$-interleaving} between $X$ and $Y$ to be a functor
\[
Z\colon \I^\delta\to \C
\]
such that $Z\circ E^0=X$ and $Z\circ E^1=Y$.
\end{definition}

Let $X(\delta)\colon \R\to \C$ be the functor obtained by shifting each object and morphism of $X$ downward by $\delta$, i.e., $X(\delta)_r:=X_{r+\delta}$ and $X(\delta)_{r,s}:=X_{r+\delta,s+\delta}$ for all $r\leq s\in \R$.  Note that a $\delta$-interleaving $Z$ between $X$ and $Y$ restricts to a pair of natural transformations $X\to Y(\delta)$ and $Y\to X(\delta)$ and that, conversely, these natural transformations fully determine $Z$; we call these natural transformations \emph{$\delta$-interleaving morphisms}.  In the case $\delta$=0, they are simply an inverse pair of natural isomorphisms between $X$ and $Y$.

\begin{definition}
We define 
\[
d_I\colon \ob \C^\RCat \times \ob \C^\RCat \to [0,\infty],
\]
the \emph{interleaving distance}, by taking  
\[
d_I(X,Y)=\inf\, \{\delta \mid X\textup{ and  }Y\textup{
    are }\delta\textup{-interleaved}\}.
\]
\end{definition}
It easy to check that if $W$ and $X$ are
$\delta$-interleaved, and $X$ and $Y$ are $\epsilon$-interleaved, then
$W$ and $Y$ are $(\delta+\epsilon)$-interleaved; it follows easily that $d_I$ is a distance on
$\ob \C^\RCat$.  Moreover, if we
have $X,X',Y\in \ob \C^\RCat $ with $X\cong X'$, then
$d_I(X,Y)=d_I(X',Y)$, so $d_I$ descends to a distance on
isomorphism classes of objects in $\C^\RCat$.

\subsection{Algebraic Stability}
The \emph{algebraic stability theorem}, a generalization of the Rips stability theorem (\cref{Thm:RST}), is a central result in the theory of persistent homology.  It was introduced by Chazal et al. \cite{chazal2009proximity},  building on earlier work by Cohen-Steiner et al. \cite{cohen2007stability} on the stability of persistent homology for $\R$-valued functions.  Since then, the result has been revisited in several papers, which have provided simpler proofs and more general formulations \cite{lesnick2015theory,chazal2012structure,bauer2015induced,bjerkevik2016stability,botnan2016algebraic,bjerkevik2021ell}.  In particular, it has been observed that the converse to the algebraic stability theorem also holds \cite{lesnick2015theory}; this is an easy consequence of the structure theorem for persistence
modules~\cite{crawley2012decomposition}.  

We state a sharp form of the algebraic stability theorem for pointwise finite dimensional (\pfd) persistence modules, as appears in \cite{bauer2015induced}:

\begin{theorem}[Forward and converse algebraic stability]\label{Thm:IsometryTheorem}
A pair of \pfd persistence modules $M$ and $N$ are
$\delta$-interleaved if and only if there exists a $\delta$-matching
between $\B M$ and $\B N$. In particular, \[d_B(\B M,\B N)=d_I(M,N).\] 
\end{theorem}

\begin{remark}\label{Rem:Not_Homotopy_Invariant}
It is easily to see that the interleaving distance $d_I$ on $\R$-spaces is stable.  A $\delta$-interleaving between $\R$-spaces $X$ and $Y$ induces a $\delta$-interleaving between $H_iX$ and $H_iY$ for all $i\geq 0$, so it follows from \cref{Thm:IsometryTheorem} that $d_I$ is homology bounding.  However, $d_I$ is not homotopy invariant: Consider $\R$-spaces $X$ and $Y$ with  $X_r=\{0\}$ and $Y_r=\R$ for all $r\in \R$, with each map $Y_r\to Y_s$ the identity on $\R$.  The inclusion $\{0\}\hookrightarrow \R$ induces an objectwise homotopy equivalence $X\hookrightarrow Y$, but $d_I(X,Y)=\infty$.  
More generally, it is easy to check that $d_I(X,Y)=\infty$ for any two $\R$-spaces $X$ and $Y$ with $\colim X$ not homeomorphic to $\colim Y$.  
\end{remark}

\subsection{Homotopy Interleavings}\label{Sec:HI}

We now introduce our homotopical generalization of interleavings.

\begin{definition}\label{def:HI}
For $\delta\geq 0$, we say $\R$-spaces $X$ and $Y$ are \emph{$\delta$-homotopy-interleaved} if there exist $\R$-spaces $X'\htp X$ and
$Y'\htp Y$ such that $X'$ and $Y'$ are $\delta$-\textup{interleaved}.
\end{definition}

\begin{definition}
The homotopy interleaving distance between $\R$-spaces $X$ and $Y$ is given by 
\[
d_{HI}(X,Y):=\inf\, \{\delta \mid X,Y\textup{ are
}\delta\textup{-homotopy-interleaved}\}.
\] 
\end{definition}

\begin{proof}[Partial Proof of \cref{Thm:Properties_of_dHI}] 
It is clear that $d_{HI}$ is symmetric and non-negative, and that for any $\R$-space $X$, $d_{HI}(X,X)=0$.  To establish that $d_{HI}$ is a distance, then, it suffices to check that $d_{HI}$ satisfies the triangle inequality; we verify this in \cref{Sec:Triangle_Ineq} below. 

It is easy to check that $d_{HI}$ is stable and homotopy invariant.  A weak equivalence between $\R$-spaces $X,Y$ induces a 0-interleaving between $H_iX, H_iY$, and as noted above, a $\delta$-interleaving between $X,Y$ induces a $\delta$-interleaving between $H_iX, H_iY$.    From these observations, the triangle inequality for $d_I$ on persistence modules, and \cref{Thm:IsometryTheorem}, we have that $d_{HI}$ is homology bounding. 
\end{proof}

\section{The Triangle Inequality for $d_{HI}$}\label{Sec:Triangle_Ineq}

In this section, we prove the triangle inequality for $d_{HI}$, thereby completing the proof of \cref{Thm:Properties_of_dHI}.  As noted in \cref{Sec:Properties}, our proof amounts to showing that some of the internal maps in a certain homotopy left Kan extension are weak homotopy equivalences.

\subsection{Generalized Interleaving Categories}
For our proof, it will be convenient to introduce a generalization of our definition of an interleaving category from \cref{Sec:Interleavings}.  
We define a \emph{marked category} to be a finite, thin category $\I$ equipped with a map $m \colon S\to [0,\infty)$, where $S$ is a subset of the set of unordered pairs of isomorphic objects in $\I$.  To simplify notation, we will write $m(\{a,b\})$ as $m(a,b)$.  We denote a pair $\{a,b\}\in S$ with $m(a,b)=\delta$ as follows:
\[
\begin{tikzcd}[column sep=5ex,row sep=2.4ex]
a \ar[bend left=20]{r} &b \ar[bend left=20,swap]{l}{\delta}
\end{tikzcd}
\]
Define $\Int{\I}$, the \emph{interleaving category} of the marked category $\I$, to be the thin category with $\obj \Int{\I}=\obj \I \times \R$ and
$\hom \Int{\I}$ generated by the set of arrows 
\begin{align*}
&\{(a,r)\to (b,r) \mid r\in \R,\ a\to b\in \hom(\I),\ \{a,b\}\not\in S\}\\
\union &\{(a,r)\to (b,r+m(a,b))\mid r\in \R,\ \{a,b\}\in S\}.
\end{align*}
Define a \emph{diagram of $\R$-spaces indexed by $\I$} to be a functor $F\colon \Int{\I}\to \Top$.  $F$ restricts to an $\R$-space $F_a$ for each $a\in \ob \I$, to a natural transformation $F_{a,b}\colon F_a\to F_b$ for each $a\to b\in \hom(\I)$ with $\{a,b\}\not \in S$, and to a $m(a,b)$-interleaving between $F_a$ and $F_b$ for each $\{a,b\}\in S$.  

\subsection{Proof of the Triangle Inequality}\label{Sec:Triangle_Ineq_Sub}

It suffices to show that if $W$ and $X$ are
$\delta$-homotopy-interleaved and $X$ and $Y$ are
$\epsilon$-homotopy-interleaved, then $W$ and $Y$ are $(\delta +
\epsilon)$-homotopy-interleaved.  Let $\A$ be the following marked category 
\[
\begin{tikzcd}[column sep=4ex,row sep=4ex]
                                                                    &                                        &\bullet\ar[swap]{dl}\ar{dr} &                                                                 &                                       \\
               \ab\ar[bend left=20]{r} &\bullet\ar[bend left=20,swap]{l}{\delta} &                                         &\bullet\ar[bend left=20]{r} &\bb\ar[bend left=20,swap]{l}{\epsilon},           
\end{tikzcd}
\]
where objects we don't (yet) wish to name explicitly are denoted by $\bullet.$  
If $W$ and $X$ are $\delta$-homotopy-interleaved and $X$ and $Y$ are $\epsilon$-homotopy-interleaved, then using the fact that $\htp$ is an equivalence relation on $\R$-spaces, there exists a diagram $F$ of $\R$-spaces indexed by $\A$ such that $F_\ab\htp W$, $F_\bb\htp Y$, and the two diagonal arrows are weak equivalences.  
By taking a cofibrant replacement of $F$, we may assume that $F$ is cofibrant.

Let $\Bb$ be the following marked extension of $\A$:
\[
\begin{tikzcd}[column sep=4ex,row sep=4ex]
 & & \bullet\ar[swap]{dl}\ar{dr} & & \\
 \ab\ar[bend left=20]{r}\ar[swap,dashed]{dr} &\bullet\ar[bend left=20,swap]{l}{\delta}\ar[dashed]{dr} & &\bullet\ar[bend left=20]{r}\ar[dashed]{dl} &\bb\ar[bend left=20,swap]{l}{\epsilon}\ar[dashed]{dl}\\
{} &\cb\ar[ bend left=20,dashed]{r} &\bullet\ar[bend left=20,dashed,swap]{l}{\delta} \ar[bend left=20,dashed]{r}&\db\ar[bend left=20,dashed,swap]{l}{\epsilon} 
\end{tikzcd}
\]
and let $\iota\colon \A\hookrightarrow \Bb$ denote the inclusion functor.  The inclusion $\iota$ induces an inclusion functor $\Int{\iota}\colon \Int{\A}\hookrightarrow \Int{\Bb}$.  

In what follows, we adopt the convention that for $G:\C\to \D$ an inclusion of categories, we write a left Kan extension $\Lan_{G}F$ as $\Lan_\D F$.  Moreover, for $d\in \ob \D$, we will write $\C\downarrow d:=G\downarrow d$ and $d\downarrow \C:=d\downarrow G$.

Since we assume $F$ to be cofibrant, $\Lan_{\Int{\Bb}} F$ computes the homotopy left Kan extension of $F$ along $\Int{\iota}$.  Since $\Int{\iota}$ is fully faithful, we have that $\Lan_{\Int{\Bb}} F\circ \Int{\iota} \cong F$ \cite[Corollary X.3.3]{mac1998categories}.  Therefore, writing $L:=\Lan_{\Int \Bb} F$, to establish that $W$ and $Y$ are $(\delta+\epsilon)$-homotopy-interleaved, it suffices to prove the following proposition; the desired $(\delta+\epsilon)$-homotopy-interleaving is then given by composition.   

\begin{proposition}\label{Prop:Morphisms_are_WEs}
The morphisms of $\R$-spaces $L_{\ab,\cb}\colon L_\ab\to L_\cb$ and $L_{\bb,\db}\colon L_\bb\to L_\db$ are weak equivalences.
\end{proposition}

\begin{proof}
We will show that $L_{\ab,\cb}$ is a weak equivalence; by symmetry, the argument for $L_{\bb,\db}$ is the same.  Let $\Bb'$ denote the marked category:
\[
\begin{tikzcd}[column sep=4ex,row sep=4ex]
 & & \eb \ar[swap]{dl}\ar{dr} & & & \\
 \ab\ar[bend left=20]{r}
&\fb\ar[bend left=20,swap]{l}{\delta}\ar{dr} & &\gb\ar[bend
  left=20]{r}\ar{dl} &\bb\ar[bend left=20,swap]{l}{\epsilon} \\
 & &\hb &
\end{tikzcd}
\]
and let $\Bb''$ denote the marked category:
\[
\begin{tikzcd}[column sep=4ex,row sep=4ex]
& & \eb \ar[swap]{dl}\ar{dr} & & & \\
 \ab \ar[bend left=20]{r}\ar[swap]{dr}
&\fb \ar[bend left=20,swap]{l}{\delta}\ar{dr} & &\gb \ar[bend
  left=20]{r}\ar{dl} &\bb \ar[bend left=20,swap]{l}{\epsilon} \\
 &\cb \ar[ bend left=20]{r} &\hb \ar[bend
  left=20,swap]{l}{\delta} & \\
\end{tikzcd}
\]
Note that we have inclusions $\A\hookrightarrow \Bb'\hookrightarrow \Bb'' \hookrightarrow \Bb$ factoring $\iota\colon \A\hookrightarrow \Bb$; these induce inclusions $\Int \A\xhookrightarrow{\iota'} \Int \Bb' \hookrightarrow \Int \Bb'' \hookrightarrow \Int \Bb$ factoring $\Int \iota \colon \Int \A\hookrightarrow \Int \Bb$.  By universality, we obtain the following diagram of left Kan extensions, commuting up to natural isomorphism:
\[
\begin{tikzcd}[column sep=40ex,row sep=4ex]
\Int \A \ar{r}{F}\ar[hookrightarrow]{d} & \Top \\
\Int \Bb' \ar[hookrightarrow]{dd}\ar[dashed]{ur}{L_{\Int \Bb'}F} \\
{}\\ 
\Int \Bb'' \ar[hookrightarrow]{d}\ar[dashed]{uuur}{L_{\Int \Bb''}L_{\Int \Bb'}F}\\
\Int \Bb \ar[dashed,swap]{uuuur}{L_{\Int \Bb} F}.
\end{tikzcd}
\]
In particular, we have $L\cong  L_{\Int \Bb}L_{\Int \Bb''}L_{\Int \Bb'}F$.  

To show that $L_{\ab,\cb}$ is a weak equivalence, we first show that $L_{\fb,\hb} \cong (L_{\Int \Bb'}F)_{\fb,\hb}$ is an objectwise acyclic cofibration.  
The key step in the argument is to show that for each $r\in \R$, the restriction of $L':=L_{\Int \Bb'} F$ to the full subcategory of $\Int \Bb'$ with the four objects \[\{(\eb,r),(\fb,r),(\gb,r),(\hb,r)\}\] is a pushout square.  

Let $\K:=\Int \A\downarrow (\hb,r)$.  By \cref{rem:Comma_Category_Special_Case}, we may identify $\K$ with the full subcategory of $\Int \A$ consisting of all objects in $\Int\A$ with a morphism to $(\hb,r)$.  Let $\hat \K$ be the full subcategory of $\J'$ with object set $ \ob(\K)\union \{(\hb,r)\}$.  Using the pointwise formula for left Kan extensions (\cref{Eq:LKE_Pointwise}), it is readily checked that the restriction of $L'$ to $\hat \K$ is naturally isomorphic to the colimit cocone of $F|_{\K}$.

A category is said to be \emph{connected} if it is non-empty and any two objects are connected by a finite zigzag of morphisms.  Note that \[\K':=(\fb,r)\leftarrow(\eb,r)\rightarrow (\gb,r)\] is a \emph{final subcategory} of $\K$, i.e., for each $d\in \ob \K$, the comma category $d\downarrow \K'$ is connected. Hence, a standard result about the preservation of colimits under final functors \cite[Section 8.3]{riehl} guarantees that the natural map $\colim F|_{\K'}\to \colim F|_\K$ is an isomorphism.  By construction, this isomorphism commutes with the maps in the colimit cocones.  It follows that the square in question is a pushout, as claimed.  

Since $F$ is cofibrant, the map $F_{(\eb,r),(\gb,r)}$ is a cofibration; by assumption it is in fact an acyclic cofibration.  Thus, \cref{Prop:Cobase_Change}\,(ii) implies that $L'_{(\fb,r),(\hb,r)}$ is an acyclic cofibration as well.  This holds for all $r$, so $L'_{\fb,\hb}$ is an objectwise acyclic cofibration.

Essentially the same argument applied to $L_{\Int \Bb''}L_{\Int \Bb'}F$ shows that for any $r\in \R$, the restriction of $L_{\Int \Bb''}L_{\Int \Bb'}F$ to the full subcategory of $\Int \Bb''$ with the four objects \[\{(\ab,r),(\fb,r-\delta),(\cb,r),(\hb,r-\delta)\}\] is a pushout square.  Then as above, since $L'_{\fb,\hb}$ is an objectwise acyclic cofibration, so is $L_{\ab,\cb} \cong (L_{\Int \Bb''}L_{\Int \Bb'}F)_{\ab,\cb}$.
\end{proof}

\begin{remark}
In the proof of \cref{Prop:Morphisms_are_WEs}, we have not used the full strength
of the hypothesis that $F$ is cofibrant, only that each internal map of $F$ is a cofibration.  However, the assumption that $F$ is cofibrant allows us to interpret $\Lan_{\Int \Bb} F$ as a homotopy Kan extension.
\end{remark}

\section{Universality of the Homotopy Interleaving Distance}\label{Sec:Universality}

In this section, we prove our universality result for $d_{HI}$,
\cref{Thm:Universality}. 

\subsection{Structure of Cofibrant Diagrams}\label{Sec:Concrete_Model}

To establish the universality of $d_{HI}$, we analyze the structure of
cofibrant objects in the projective model structure on diagrams of spaces indexed by directed sets.

For any small category $\I$, we say a functor $X\colon \I\to \Top$ is
a \emph{closed filtration} if each of the internal maps $X_{a,b}\colon
X_a\to X_b$ is a closed inclusion.  

\begin{proposition}\label{Lem:CofibrantReplacementsAreFiltrations}
Any cofibrant diagram in $\Top^\I$ is a closed filtration.  
\end{proposition}

\begin{proof}
As noted in \cref{Ex:Projective_Model}, each internal map in a
cofibrant diagram in $\Top$ is a cofibration, and hence a closed
inclusion~\cite[Lemma 2.4.6, Theorem 2.4.23, Lemma 2.4.25]{hovey}.
\end{proof}

\begin{lemma}[{\cite[Corollary 2.23, Proposition 2.35]{strickland2009category}}]\label{Lem:PushoutLemma}
Given a pushout square
\[
\begin{tikzcd}
  A \ar[]{r} \ar["f",swap]{d} & X\ar["g"]{d} \\
  B \ar[]{r} & Y,
\end{tikzcd}
\]
in the category $\Top$ of CGWH spaces with $f$ a closed inclusion, the square is also a pushout in the category of all topological spaces and $g$ is also a closed inclusion.
\end{lemma} 

We define a \emph{directed set} to be a non-empty poset $\I$ such that for all $a,b\in \I$,
there exists $c\in \I$ with $a,b\leq c$.  Given a directed set $\I$, a
functor $X\colon \I\to \Top$, and $a\in \I$, let
$\maptocolim{X}{a}\colon X_a\to \colim X$ denote the canonical map.

\begin{lemma}[{\cite[Corollary 2.23, Lemma 3.3]{strickland2009category}}]\label{Lem:ColimitEmbeddingLemma}\label{Cor:CoconesOfCofReplacements}  
If $\I$ is a directed set and $X\colon \I\to \Top$ is a closed
filtration, then $\colim X$ is the colimit in the category of all
topological spaces, and each map $\maptocolim{X}{a}$ is a closed inclusion.
\end{lemma} 

Adapting terminology introduced in \cite{carlsson2010computing}, for
$\I$ a directed set, we say a functor $X\colon \I \to \Top$ is
\emph{1-critical} if $X$ is a closed filtration and for each $x\in
\colim X$, the set  
\[
\{a\in \I \mid x\in \im \maptocolim{X}{a}\}
\] 
has a minimum element.  We then have a function $\fc^X\colon \colim
X\to \I$ sending each $x\in \colim X$ to this minimum element.

\begin{proposition}\label{Prop:CofReplacementIs1Critical}
For any directed set $\I$, each cofibrant diagram in $\Top^\I$ is 1-critical.
\end{proposition}

\begin{proof}
Recall that the projective model structure on $\Top^\I$ is compactly generated, where the generating cofibrations are free diagrams on  the inclusions $S^{n-1}\hookrightarrow D^n$.  Thus, by \cref{prop:cell-complex}, the cofibrant diagrams $Y$ in $\Top^\I$ are exactly the retracts of diagrams $\colim Z$, where $Z:\mathbb N\to \Top^{\I}$ is given by taking $Z_0=\emptyset$ and each $Z_{i+1}$ is a pushout of $Z_i$ along of a product of generating cofibrations.  

To show that such $Y$ is 1-critical, we first show that each $Z_i$ is 1-critical, and then that $\colim Z$ is 1-critical.  Note that the diagrams $Z_i$, $\colim Z$, and $Y$ are all cofibrant.  Since any cofibrant diagram $D$ is a closed filtration by \cref{Lem:CofibrantReplacementsAreFiltrations}, checking 1-criticality of $D$ amounts to checking that every element of $\colim D$ has a minimum birth index.  

Let us write $W_i:=\colim Z_i$.  For each $i\geq 0$, we have an induced map $w_i:W_i\to W_{i+1}$, and these define a functor $W:\mathbb N\to \Top$.  Each $w_i$ is a pushout along a closed inclusion, hence by \cref{Lem:PushoutLemma}, $w_i$ is itself a closed inclusion.  Thus, $W$ is a closed filtration.  Any closed filtration indexed by $\mathbb N$ is 1-critical, so $W$ is 1-critical.

We show that each $Z_i$ is 1-critical by induction: Clearly $Z_0$ is 1-critical.  Since any free diagram in $\Top^\I$ is 1-critical, it follows from \cref{Lem:PushoutLemma} that if $Z_i$ is 1-critical, then $Z_{i+1}$ is 1-critical.  Moreover, for any $x\in \colim Z_{i}$, we have
\begin{equation}\label{eq:birth_doesn't_change}
\fc^{Z_i}(x)=\fc^{Z_{i+1}}(w_i(x)).
\end{equation}

To see that $\colim Z$ is 1-critical, let $x\in \colim \colim Z$.  For any small category $\C$ and cocomplete category $\D$, the functor $\colim: \D^\C\to \D$ is a left adjoint and hence preserves colimits \cite[Proposition 4.5.1, Theorem 4.5.3]{riehl2017category}.  Taking $\C=\mathbb N$ and $\D=\Top^\I$, this implies that $\colim \colim Z=\colim W$.  Let $j=\fc^{W}(x)$.  By \cref{Lem:ColimitEmbeddingLemma}, the map $\mu^{W}_j:W_j\to \colim W$ is an injection; let $y=(\mu^{W}_j)^{-1}(x)$.  Using \cref{eq:birth_doesn't_change}, it is easily checked that \[\fc^{Z_j}(y)=\min  
\,\{a\in \I \mid x\in \im \maptocolim{\colim Z}{a}\}.
\] 
Since $x$ was chosen arbitrarily, this implies that $Z$ is 1-critical.  

Finally, using the fact that the colimit of a retract of diagrams is a retract, it is straightforward to check that the retract of a 1-critical diagram is 1-critical.  Thus, the 1-criticality of $Z$ implies the 1-criticality of $Y$.  \qedhere
\end{proof}

\begin{remark}
An analogue of \cref{Prop:CofReplacementIs1Critical} for simplicial sets is given in \cite[Proposition 4.5]{lanari2020rectification}.
\end{remark}

We next define a category $\Funs$ whose objects are functions $\gamma_T\colon T\to \I$ such that $T\in \ob \Top$.  We take $\hom_{\Funs}(\gamma_S,\gamma_T)$ to be the set of continuous functions $f\colon S\to T$ such that $\gamma_T\circ f\leq \gamma_S$.  We emphasize that when $\I$ happens to carry a topology (e.g., when $\I=\R$), we do not require $\gamma_T\in \ob \Funs$ to be continuous, but we do require morphisms in $\Funs$ to be continuous. 

Let $\Top^{\I}_{\crit}$ denote the full subcategory of $\Top^{\I}$ whose objects are the 1-critical diagrams.  The functoriality of colimits tells us that for diagrams $X,Y\colon \I\to \Top$, a natural transformation $f\colon X\to Y$ induces a map \[\colim f\colon \colim X\to \colim Y.\]   If $X,Y$ are 1-critical, then $\fc^{Y}\circ \colim f \leq \fc^{X}$.  We thus have a functor \[\acolim\colon \Top^{\I}_{\crit}\to \Funs\] which sends each 1-critical diagram $X$ to $\fc^X\colon \colim X\to \I$.  

We also have an obvious functor $\Sb\colon \Funs\to \Top^\I$ with 
\[\Sb(\gamma_T)_a:=\{y\in T\mid \gamma_T(y)\leq a\}.\]  This generalizes the sublevelset filtration construction introduced in \cref{Sec:Intro_Homotopy_Interleavings}.

\begin{proposition}\label{Prop:ImportantEquivalence}
If $\I$ is a directed set, then $\Sb \comp \acolim\cong \id_{\Top^{\I}_{\crit}}$.  
\end{proposition}

\begin{proof}
Consider a diagram $X\in \ob \Top^\I_{\crit}$.  For $a\in \I$, \[({\Sb}\circ \acolim X)_a=\im \maptocolim{X}{a}.\]  By \cref{Cor:CoconesOfCofReplacements}, $\maptocolim{X}{a}$ is a homeomorphism onto its image.  For $a\leq b\in \I$ we have $\maptocolim{X}{b}\circ X_{a,b}=\maptocolim{X}{a}$ so these homeomorphisms define a natural isomorphism $\mu^X\colon  X\to {\Sb}\circ \acolim X$.  Further, the natural isomorphisms $\{\mu^X\}_{X\in \Top^{\I}_{\crit}}$ are natural in $X$, so this collection assembles into a natural isomorphism $\id_{\Top^{\I}_{\crit}}\to {\Sb}\circ \acolim$.
\end{proof}

\subsection{Proof of Universality}\label{Sec:Universality_Sub}
The main step in our proof that $d_{HI}$ is universal is the following:
\begin{proposition}\label{Prop:LiftToFunctions}
For any $\delta$-interleaved $\R$-spaces $X$, $Y$, there exists a topological space $T$ and functions
$\gamma^X,\gamma^Y\colon  T\to \R$ such that $\Sb(\gamma^X)\htp X$, $\Sb(\gamma^Y)\htp Y$, and $d_\infty(\gamma^X,\gamma^Y)\leq \delta$.
\end{proposition}

\begin{proof}
For $\I$ a small category, let $\Q$ denote a cofibrant replacement
functor in the projective model structure on $\Top^\I$.

It will be convenient for us to treat the cases $\delta=0$ and $\delta>0$ separately.  First, let $\delta=0$, so that we have an isomorphism $X\to Y$.  We take $T=\colim \Q X$.  By \cref{Prop:CofReplacementIs1Critical}, $\Q X$ is 1-critical.  We let $\gamma^X=\gamma^Y=\fc^{\Q X}$.  Since $\R$, together with its total order, is a directed set, by \cref{Prop:CofReplacementIs1Critical} and \cref{Prop:ImportantEquivalence} imply that $\Sb(\gamma^Y)=\Sb(\gamma^X)\cong \Q X$.  We also have a weak equivalence $\Q X\to X$.  Composing, we thus obtain weak equivalences $\Sb(\gamma^X)\to X$, $\Sb(\gamma^Y)\to Y$, as desired.  This completes the proof in the case $\delta=0$.

Now assume $\delta>0$.  Recall the definitions of the interleaving category $\I^\delta$ and the functors $E^0,E^1\colon \RCat\to \I^\delta$ from \cref{Sec:Interleavings}, and note that when $\delta>0$, $\I^\delta$ is a poset category; in fact the underlying poset is a directed set.

Since $X$ and $Y$ are $\delta$-interleaved, there exists a functor $Z\colon \I^\delta\to \Top$ such that $Z\circ E^0=X$ and $Z\circ E^1=Y$.  We define $T:=\colim \Q Z$.  

%


$E^0$ and $E^1$ are both final functors \cite[Section 8.3]{riehl}.  Hence, we have canonical identifications of $\colim(\Q Z\circ E^0)$  and $\colim(\Q Z\circ E^1)$ with $T$ such that for each $r\in \R$,
\[\maptocolim{\Q Z}{(r,0)}=\maptocolim{\Q Z\circ E^0}{r},\qquad\maptocolim{\Q Z}{(r,1)}=\maptocolim{\Q Z\circ E^1}{r}.\]

We claim that $\Q Z\circ E^0$ and $\Q Z\circ E^1$ are each 1-critical.  We show this for $\Q Z\circ E^0$; the proof for $\Q Z\circ E^1$ is the same.  First note that $\Q Z$ is 1-critical by \cref{Prop:CofReplacementIs1Critical}.  In particular, for each $a\in \I^\delta$, $\maptocolim{\Q Z}{a}\colon Z_a\to T$ is a closed inclusion.  Therefore, for each $r\in \R$, $\maptocolim{\Q Z\circ E^0}{r}\colon (\Q Z\circ E^0)_r\to T$ is a closed inclusion.

Since $\Q Z$ is 1-critical, for each $z\in T$ there is a minimum element $(r,j)\in \ob \I^\delta$ such that $z\in \im \mu^{\Q Z}_{(r,j)}$.  We then have
\[r+j\delta=\min\, \{s\in \R \mid z\in \im \maptocolim{\Q Z\circ E^0}{s}\}.\]  
It follows that $\Q Z\circ E^0$ is 1-critical.

Since $\Q Z\circ E^0$ and $\Q Z\circ E^1$ are each 1-critical, we may define $\gamma^X,\gamma^Y\colon T\to \R$ by
\begin{align*}
\gamma^X&=\fc^{\Q Z\circ E^0},\\
\gamma^Y&=\fc^{\Q Z\circ E^1}.
\end{align*}
By \cref{Prop:ImportantEquivalence}, we have
\[\Sb(\gamma^X)\cong \Q Z\circ E^0,\qquad \Sb(\gamma^Y)\cong \Q Z\circ E^1.\]  By construction, there exists a weak equivalence $\Q Z\to Z$; restricting, we obtain  weak equivalences $\Q Z\circ E^0\to X$ and $\Q Z\circ E^1\to Y$.  Hence, there exist weak equivalences $\Sb(\gamma^X)\to X$ and $\Sb(\gamma^Y)\to Y$.

It remains to check that $d_\infty(\gamma^X,\gamma^Y )\leq \delta$.  Consider $z\in T=\colim \Q Z$.  There is a minimum index $(r,j)\in \ob \I^\delta$ such that 
$z\in \im \maptocolim{\Q Z}{(r,j)}$.  If $j=0$ then $\gamma^X(z)=r$ and $\gamma^Y(z)=r+\delta$.  If on the other hand $j=1$, then $\gamma^Y(z)=r$ and $\gamma^X(z)=r+\delta$.  Clearly, in either case, we have 
$\|\gamma^X(z)-\gamma^Y(z))\|_\infty=\delta$. 
Since this holds for all $z\in T$ we have $d_\infty(\gamma^X,\gamma^Y) \leq \delta$ as desired (with strict equality unless $\colim\Q Z=\emptyset$).
\end{proof}

\begin{proof}[Proof of \cref{Thm:Universality}]
Let $X$ and $Y$ be $\R$-spaces with $d_{HI}(X,Y)=\delta$.  Then for all $\delta' > \delta$, $X$ and $Y$ are $\delta'$-homotopy interleaved, i.e., there exist $\delta'$-interleaved $\R$-spaces $X'$, $Y'$ with $X'\htp X$ and $Y'\htp Y$.  \cref{Prop:LiftToFunctions} gives us a topological space $T$ and functions $\gamma^{X'},\gamma^{Y'}\colon T\to \R$ with $\Sb(\gamma^{X'})\htp X$, $\Sb(\gamma^{Y'})\htp Y$, and $d_\infty(\gamma^{X'},\gamma^{Y'})\leq \delta'$.  

Suppose $d$ is a stable, homotopy invariant distance on $\R$-spaces.  Then by stability, $d(\Sb(\gamma^{X'}),\Sb(\gamma^{Y'}))\leq \delta'$.  Therefore, by homotopy invariance and the triangle inequality for $d$, we have $d(X,Y)\leq \delta'$.  Since this holds for arbitrary $\delta'>\delta$ we have that $d(X,Y)\leq \delta=d_{HI}(X,Y)$.
\end{proof}


\section{Applications}\label{Sec:Applications}
In this section, we show that several fundamental TDA theorems can be formulated on the space level using the homotopy interleaving distance, or any distance satisfying our stability and homotopy invariance axioms.  Specifically, we prove \cref{Prop:Two_Axioms}, and present in detail the results mentioned in \cref{Sec:Other_Strengthenings}.

\subsection{Stability of Rips Filtrations}\label{Sec:Rips_Application}
\label{Rips_Stability_Refinement}

Recall that~\cref{Prop:Two_Axioms}, which strengthens the Rips stability theorem (\cref{Thm:RST}) to a purely homotopy-theoretic result,   
says the following: If $d$ is any distance on $\R$-spaces satisfying the
stability and homotopy invariance axioms of 
\cref{Def:Two_Axioms},
then for all metric spaces $P$ and $Q$, we have 
\[
d(\Rips(P),\Rips(Q))\leq d_{GH}(P,Q).
\]
We are aware of four different proofs of the Rips stability theorem.  The original proof of Chazal et al.~\cite{chazal2009gromov} isometrically embeds the metric spaces into a space of $\R$-valued functions with the sup-norm distance and applies the nerve theorem.   A proof of \cref{Prop:Two_Axioms} is already implicit in this proof.  A later proof appearing in work of Chazal et al. \cite{chazal2014persistence} avoids use of embeddings and the nerve theorem, and instead considers multi-valued maps between simplicial complexes.  An elegant third proof, due to M\'emoli~\cite{memoli2017distance}, relies on Quillen's Theorem A for simplicial complexes~\cite[Page 93]{quillen1973higher}.  A fourth proof, appearing in the Ph.D. thesis of Scoccola \cite{scoccola2020locally} and building on some of our ideas, also uses Quillen's Theorem A, but frames the argument 
in terms of an abstract principle about the preservation of generalized homotopy interleavings under certain enriched functors.

We will verify~\cref{Prop:Two_Axioms} by following M\'emoli's proof of the Rips stability theorem.  
We first review the definition of the Gromov-Hausdorff distance.  Given sets $S,T$, a \emph{correspondence} between $S$ and $T$ is a set $C\subset S\times T$ such that the coordinate projections $\proj_S\colon C\to S$ and $\proj_T\colon C\to T$ are surjections.  Let $\Gamma(S,T)$ denote the set of all correspondences between $S$ and $T$.  

\begin{definition}\label{Def:GH_Dist}
The Gromov-Hausdorff distance between metric spaces $P$ and $Q$ is given by
\[d_{GH}(P,Q)=\frac{1}{2}\, \inf_{C\in\Gamma(P,Q)}\ \ \sup_{(p,q),(p',q')\in C}\  |d_P(p,p')-d_Q(q,q')|.\]
This defines a distance $d_{GH}$ (i.e., extended pseudometric) on arbitrary metric space, which restricts to a genuine metric on compact metric spaces.  
\end{definition}

\begin{proof}[Proof of \cref{Prop:Two_Axioms}]
If $P$ and $Q$ are metric spaces and $\delta>d_{GH}(P,Q)$, then there exists a correspondence $C\subset P\times Q$ with $|d_P(p,p')-d_Q(q,q')|\leq 2\delta$ for all $((p,q),(p',q'))\in C$.

Let $[C]$ denote the maximal simplicial complex with vertices $C$, i.e., $[C]$ consists of all non-empty finite subsets of $C$.  We define a simplicial filtration $\F^P$ on $[C]$ by taking $\sigma\in \F^P_r$ if and only if
\[d_P(\proj_P(u),\proj_P(v))\leq 2r\] for all $u,v\in \sigma$. 

 $\F^P$ induces a function $\gamma^P\colon [C]\to \R$ which sends a simplex $\sigma$ to the minimum $r\in \R$ such that $\sigma\in \F^P_r$.   Note that $\F^P$ is equal to the simplicial sublevelset filtration $\Sb(\gamma^P)$.
Define a simplicial filtration $\F^Q$ and function $\gamma^Q\colon [C]\to \R$ analogously.

By the way we chose $C$, we have that  
\[
d_\infty(\gamma^P,\gamma^Q)\leq \delta.
\]
Thus, by the stability axiom for $d$ and the fact that 
\[\Sb(\gamma^P)= \F^P, \quad \Sb(\gamma^Q)= \F^Q,\]
we have that $d(\F^P,\F^Q)\leq \delta$. 

Quillen's theorem A for simplicial complexes \cite{quillen1973higher} says that if $f\colon S\to T$ is a simplicial map of simplicial complexes such that $f^{-1}(\sigma)$ is contractible for each (closed) simplex $\sigma \in T$, then $f$ is a homotopy equivalence.  Note that $\proj_P\colon C\to P$ induces a morphism $g\colon \F^P\to \Rips(P)$.  By Quillen's theorem A for simplicial complexes, $g$ is an objectwise homotopy equivalence.  Symmetrically, $\proj_Q\colon C\to Q$
induces an objectwise homotopy equivalence $\F^Q\to\Rips(Q)$.  Thus by the homotopy invariance of
$d$, \[d(\Rips(P),\F^P)=d(\Rips(Q),\F^Q)=0.\]

By the triangle inequality for $d$, we have \[d(\Rips(P),\Rips(Q))\leq \delta.\]
Since this holds for all $\delta>d_{GH}(P,Q)$, it follows that
\[d(\Rips(P),\Rips(Q))\leq d_{GH}(P,Q),\]
as desired.
\end{proof}

\subsection{Stability of Simplicial Filtrations}\label{Sec:Stabilty_Correspondences}
We next observe that the general stability result for simplicial filtrations appearing in \cite{chazal2014persistence} can also be cast in our axiomatic framework, by essentially the same argument as in the proof of \cref{Prop:Two_Axioms}.

\begin{definition}
Given simplicial filtrations $X,Y:\R\to \mathbf{Simp}$, let $P$ and $Q$ be the vertex sets of $\colim(X)$ and $\colim(Y)$ respectively.  A correspondence $C$ between $P$ and $Q$ is said to be \emph{$\delta$-simplicial} if the following hold for any $r\in \R$: 
\begin{enumerate}
\item For any simplex $\sigma\in X_r$, every finite non-empty subset of $\proj_{Q}(\proj_{P}^{-1}(\sigma))$ is a simplex in $X_{r+\delta}$,
\item For any simplex $\sigma\in Y_r$, every finite non-empty subset of $\proj_{P}(\proj_{Q}^{-1}(\sigma))$ is a simplex in $Y_{r+\delta}$,
\end{enumerate}
\end{definition}

\begin{example}[{\cite[Lemma 4.3]{chazal2014persistence}}]\label{Ex:Rips_Simplicial}
If $P$ and $Q$ are finite metric spaces, then for any $\delta>d_{GH}(P,Q)$ there exists a correspondence between $P$ and $Q$ such that $\sup_{(p,q),(p',q')\in C}\  |d_P(p,p')-d_Q(q,q')|\leq 2\delta$.  $C$ is a $\delta$-simplicial correspondence between $\Rips(P)$ and $\Rips(Q)$.  
\end{example}

The following is proven in \cite{chazal2014persistence} using the language of multi-valued maps and a contiguity argument.

\begin{proposition}[{\cite[Proposition 4.2]{chazal2014persistence}}]\label{Prop:Chazal_Multi_Valued} If $C$ is $\delta$-simplicial, then for each $i\geq 0$, $H_i(X)$ and $H_i(Y)$ are $\delta$-interleaved.  
\end{proposition}

A slight variant of our proof of \cref{Prop:Two_Axioms} establishes the following filtration-level formulation of \cref{Prop:Chazal_Multi_Valued}.

\begin{proposition}\label{Prop:Chazal_Lift} If $C$ is $\delta$-simplicial, then for any stable and homotopy invariant distance $d$ on $\R$-spaces, we have $d(X,Y)\leq \delta$. \end{proposition}

In view of \cref{Ex:Rips_Simplicial}, \cref{Prop:Chazal_Lift} implies \cref{Prop:Two_Axioms}.  As mentioned in the introduction, \cref{Prop:Chazal_Lift} also yields filtration-level formulations of stability results for \v Cech, Dowker, and Witness complexes, via arguments appearing in \cite[Section 4]{chazal2014persistence}.

\begin{remark}
The proof of \cref{Prop:Two_Axioms} also adapts readily to show that if $C$ is $\delta$-simplicial, then $X$ and $Y$ are $\delta$-homotopy interleaved, which strengthens \cref{Prop:Chazal_Multi_Valued}.  Note, however, that as stated, \cref{Prop:Chazal_Lift} does not quite strengthen \cref{Prop:Chazal_Multi_Valued}, since the latter guarantees the existence of a $\delta$-interleaving, not just a bound on the interleaving distance.  This is an artifact of the way we've chosen to set up our axiomatic framework; a variant of our framework can be given which axiomatically treats $\delta$-homotopy interleavings, rather than the homotopy interleaving distance, and thus provides an axiomatic strengthening of \cref{Prop:Chazal_Lift}.  We have opted to instead work with our distance-based approach because it is less technical.
\end{remark}

\subsection{Sparse Approximation of \v Cech and Rips Filtrations}\label{Sec:Sparse_Approximation}
Vietoris-Rips filtrations grow large very quickly as the number of points grows: For a finite metric space $P$, $\Rips(P)$ is a filtration of the full simplex with vertices $P$, which has $2^{|P|}-1$ faces.  Even to compute only $1^{\mathrm{st}}$ persistent homology, we need the 2-skeleton of $\Rips(P)$, which has ${{|P|}\choose{3}}=O(|P|^3)$ faces.  Since we are often interested in computing the persistent homology of finite metric spaces with tens thousands of points, this size can be prohibitively large.  
This motivates the search for a smaller filtration whose persistent homology is (exactly or approximately) the same as that of $\Rips(P)$.  There is a large literature on this.  In early work on TDA, landmark-based constructions called \emph{witness filtrations} were used as surrogates for Vietoris-Rips or \v Cech filtrations \cite{chan2013topology,carlsson2008local,de2004topological,chazal2008towards,boissonnat2009manifold}. 

A key advance was made in work of Sheehy \cite{sheehy2013linear} which, for finite metric spaces $P$ of constant doubling dimension, introduced a simplicial filtration whose size is linear in $|P|$ and whose persistent homology is a provably good approximation to that of $\Rips(P)$.   Sheehy's approach is based on the idea of \emph{hierarchical net trees} from computational geometry, and can be interpreted as an adaptive landmarking strategy, where landmarks are removed as the scale parameter of the Rips complex increases.  Subsequently, Cavanna, Jahanseir, and Sheehy \cite{cavanna2015geometric} introduced a simpler and more intuitive geometric variant of Sheehy's sparsification construction, which provides sparse approximations of both \v Cech and Rips filtrations.
In the years following Sheehy's initial work, many other sparse approximation constructions have appeared, applying to different filtration types (e.g., \v Cech, Delaunay, and Dowker filtrations) and offering different tradeoffs between filtration size, approximation bounds, and practical performance \cite{choudhary2021improved,brehm2018sparips,sheehy2021sparse,choudhary2019improved,botnan2015approximating,brun2019sparse,choudhary2019polynomial,dey2019simba,dey2014computing}.  

In nearly all of these works, the results are formulated directly in terms of barcodes or persistence modules, not on the level of filtrations.  Yet the arguments always proceed at the level of filtrations and involve some homotopy theoretic step, such as an application of the nerve theorem or a contiguity argument.  We hypothesize that all of these results lift without serious difficulty to the level of filtrations, using the language of homotopy interleavings.  Here, we verify this hypothesis for just one sparse approximation result, namely the geometric sparse filtration of \cite{cavanna2015geometric}.  In fact, we will observe that the argument given in \cite{cavanna2015geometric} lifts almost immediately to the level of filtrations.

\subsubsection*{Persistent Nerve Theorem}
The essential homotopy theoretic ingredient will be a version of the \emph{persistent nerve theorem}; we begin by discussing this.  The standard formulation of the persistent nerve theorem in the TDA literature \cite{chazal2008towards} concerns open covers.  However, 
to analyze the sparse nerve construction of \cite{cavanna2015geometric}, it is more convenient to work with a version of the persistent nerve theorem for closed, convex covers, \cref{thm:persistentnerve} below. This version is TDA folklore, but a proof appears in a recent paper by Bauer et al. \cite{bauer2022unified}, which also proves several other versions of the persistent nerve theorem. 

\begin{definition}[Good Cover of a Diagram of Spaces]\label{Def:_Cover_of_Filtration}~
Let $\I$ be a small category.  A \emph{cover} of a functor $X:\I\to \Top$ indexed by a set $S$ is a collection of functors 
\[U=\{U^s:\I\to \Top\}_{s\in S}\] such that for each $a\in \I$, $\{U^s_a \mid s\in S\}$ is a cover of $X_a$.  We say $U$ is \emph{good} if $X$ takes values in a fixed Euclidean space, $S$ is finite, and $U^s_a$ is closed and convex for all $s\in S$ and $a\in \I$.
\end{definition}

The usual definition of a nerve extends immediately to a \emph{nerve} $\Ner\left(U\right):\I\to \mathbf{Simp}$ of any cover $U$ of a functor $X:\I\to \Top$.
 
\begin{theorem}[Persistent Nerve Theorem~\cite{bauer2022unified}]\label{thm:persistentnerve}
If $U$ is a good cover of $X:\I\to \Top$, then $X$ and $\Ner(U)$ are weakly equivalent.
\end{theorem}

\begin{example}\label{Ex:Cech}
Let us fix $p\in [1,\infty]$.  For $x\in \R^n$ and $r\in \R$, let $B(x,r)$ be the closed $\ell^p$-ball of radius $r$ centered at $x$, i.e., \[B(x,r)=\{y\in \R^n\mid \|x-y\|_p\leq r\}.\] 
For $P\subset \R^n$, consider the \emph{offset filtration} $O(P):[0,\infty)\to \Top$ given by 
\[O(P)_r=\bigcup_{x\in P} B(x,r).\]
Note that if $P$ is compact, then $O(P)=\Sb(d_P)$, where $d_P:\R^n\to \R$ is the distance function to $P$.  

The set of filtrations $U:=\{O(\{x\})\}_{x\in P}$ is a cover of $O(P)$, and if $P$ is finite, it is a good cover.  We call 
$\Ner(U)$ the \emph{\v Cech filtration} of $P$ and denote it as $\Cech(P)$.  \cref{thm:persistentnerve} tells us that if $P$ is finite, then $\Cech(P)\simeq O(P)$.
\end{example}

\subsubsection*{The Sparse Filtration Construction} We now explain the sparse filtration construction of \cite{cavanna2015geometric}.    
Given $P\subset \R^n$ finite and a parameter $\epsilon\in (0,1)$, \cite{cavanna2015geometric} defines a simplicial filtration $S^\epsilon(P):[0,\infty)\to \mathbf{Simp}$ whose size is linear in $|P|$ for $P$ of constant doubling dimension; the parameter  $\epsilon\in (0,1)$ controls the sparsity of the construction. 

To construct $S^\epsilon(P)$, we first choose a \emph{greedy permutation} of $P$, i.e., an ordering $P=\{p_1,\dots,p_m\}$ such that for all $i\geq 2$, $p_i$ is a furthest point from $\{p_1,\dots,p_{i-1}\}$.  Let $\lambda_1=\infty$, and for $i\in \{2,\ldots,m\}$, let $\lambda_i=\min_{j<i} d_P(p_i,p_j)$.

For $r\in [0,\infty)$, let 
\begin{align*}
b_i(r)&=
\begin{cases}
B(p_i,r)&\textup{ if $0\leq r\leq \lambda_i(1+\epsilon)/\epsilon$},\\
B(p_i,\lambda_i(1+\epsilon)\epsilon)&\textup{ if $\lambda_i(1+\epsilon)/\epsilon\leq r\leq \lambda_i(1+\epsilon)^2/\epsilon$},\\
\emptyset&\textup{ otherwise},
\end{cases}\\
U^i_r&=\bigcup_{0\leq q\leq r} b_i(q)\times \{q\}\subset \R^{n+1}.
\end{align*}

For $r\in [0, \lambda_i(1+\epsilon)/\epsilon]$, $U^i_r$ is an $\ell^p$-cone, and for $r>\lambda_i(1+\epsilon)/\epsilon$, $U^i_r$ is the union of an $\ell^p$-cone and an $\ell^p$-cylinder whose intersection is a common face.  For all $r\in [0,\infty)$, $U^i_r$ is closed and convex.

The spaces $U^i_r$ assemble into a filtration $U^i:[0,\infty)\to \R$.  Let $U=\{U^i\}_{i\in [m]}$.  We define the filtration $S^\epsilon(P)$ to be $\Ner(U).$  

\subsubsection*{Approximation Guarantee}
It is shown in \cite{cavanna2015geometric} that the bottleneck distance between the barcodes of $S^\epsilon(P)$ and $\Cech(P)$ are close on the log scale.  To give the precise statement, let $\exp\colon \R\to (0,\infty)$ denote the exponential function $x\mapsto 2^x$.
 
\begin{theorem}[\cite{cavanna2015geometric}]\label{Thm:Sheehy_Sparse}
For any finite $P\subset \R^n$ and $i\geq 0$,
\[d_B(\B{}_i(S^\epsilon(P)\comp\exp),\B{}_i(\Cech(P) \comp \exp))\leq\log(1+\epsilon).\]
\end{theorem}

We now observe that \cref{Thm:Sheehy_Sparse} strengthens to an approximation result on the filtration level with essentially the same proof. 

\begin{theorem}\label{Thm:Sparse_Refinement}
For any stable, homotopy invariant distance $d$ on $\R$-spaces and finite $P\subset \R^n$,
\[d(S^\epsilon(P)\comp \exp,\Cech(P)\comp \exp)\leq \log(1+\epsilon).\]
\end{theorem}

\begin{proof}[Proof of \cref{Thm:Sparse_Refinement}]
Consider the filtration $X:[0,\infty)\to \Top$ given by \[X_r=\bigcup_{i\in [m]} U^i_r.\]
$U$ is clearly a good cover of $X$.  Hence, $S^\epsilon(P)\simeq X$ by \cref{thm:persistentnerve}.

Next, define a filtration $Y:[0,\infty)\to \Top$ by 
\begin{align*}
Y^i_r&=
\begin{cases}
B(p_i,r)&\textup{ if $0\leq r\leq \lambda_i(1+\epsilon)/\epsilon$,}\\
B(p_i,\lambda_i(1+\epsilon)\epsilon)&\textup{ if $\lambda_i(1+\epsilon)/\epsilon\leq r$.}
\end{cases}\\
Y_r&=\bigcup_{i\in [m]} Y^i_r
\end{align*}
For $r\in [0,\infty)$, consider the coordinate hyperplane \[H_r=\{x\in \R^{n+1}\mid x_{n+1}=r\}.\]  It is shown in \cite[Corollary 2]{cavanna2015geometric} that  $X\cap H_r=Y_r\times \{r\}$.  It follows that the projection $\R^{n+1}\to \R^n$ onto the first $n$ coordinates restricts to a continuous surjection $\rho_r: X_r\to Y_r$.  In fact $\rho_r$ is a homotopy equivalence, because $X_r$ deformation retracts onto $X\cap H_r$ via a straight-line homotopy.  Moreover, the maps $\rho_r$ assemble into an objectwise homotopy equivalence $X\to Y$.  By \cref{thm:persistentnerve}, we also have $O(P)\simeq \Cech(P)$.  It is proven in \cite[Corollary 2]{cavanna2015geometric} that $Y\circ\exp$ and $O(P)\circ \exp$ are $\log(1+\epsilon)$-interleaved via subspace inclusions.  Given this, one can use \cref{Prop:ImportantEquivalence} to check that $Y\circ\exp$ and $O(P)\circ\exp$ are the respective sublevel filtrations of functions $f,g:\R^n\to \R$ such that $\|f-g\|_\infty \leq \log(1+\epsilon)$.  The result now follows from the stability and homotopy invariance of $d$.
\end{proof}

\begin{remark}
Via the Kuratowski embedding, any finite metric space $P$ embeds isometrically into $\R^{|P|}$ with the $\ell^\infty$ metric, where Rips and \v Cech filtrations are equal \cite{chazal2009gromov}.  Hence, the sparse approximation of the \v Cech filtration considered above also provides a sparse approximation of the Rips filtration.  
\end{remark}

\subsection{A Weak Law of Large Numbers for Filtrations}\label{Sec:WLLN}
As a final application of our axioms for distances on $\R$-spaces, we prove a simple weak law of large numbers for \v Cech filtrations.  We use a standard union bound argument.  

Let $M\subset \R^n$ be an $m$-dimensional compact Riemannian manifold and let $d$ denote the Euclidean metric on $\R^n$.  We regard $M$ as a probability space with respect to the normalized $m$-dimensional Hausdorff measure $\mu$.  Let $P^m$ be an i.i.d. sample of $M$ of size $m$.  

\begin{proposition}\label{Prop:WLLN}
For any stable and homotopy invariant distance $d$ on $\R$-spaces, $d(\Cech(P^m),O(M))$ converges in probability to 0 as $m\to \infty$.
\end{proposition}

\begin{proof}
Let $\epsilon>0$.  We need to show that \[\lim_{n\to \infty} \mathbb{P}(d(\Cech(P^m),O(M))> \epsilon)=0.\]  Consider the cover of $M$ by open balls of radius $\epsilon/2$.  Since $M$ is compact, there exits a finite subcover $U^1,\ldots, U^l$ of $M$.  Let \[c=\min(\mu(U^1),\ldots,\mu(U^l)),\] and note that $c>0$.  Let $E_m^i$ be the event that $P^m\cap U^i=\emptyset$, and let \[E_m=\bigcup_{i=1}^l E_m^i.\]  In the complement of the event $E_m$, we have $d_H(P^m,M)\leq\epsilon$, where $d_H$ denotes the Hausdorff distance.  Thus, using the notation for distance functions of \cref{Ex:Cech}, we have $\|d_{P^m}-d_{M}\|_{\infty}\leq \epsilon$.  By the stability of $d$, this implies that $d(O(P^m),O(M))\leq \epsilon$.  The persistent nerve theorem (\cref{thm:persistentnerve}) tells us that $\Cech(P^m)\simeq O(P^m)$.  Hence, by the homotopy invariance of $d$, in the complement of the event $E_m$ we have that $d(\Cech(P^m),O(M))\leq \epsilon$.  Thus, \[\mathbb{P}(d(\Cech(P^m),O(M))> \epsilon)\leq \mathbb{P}(E_m)\leq \sum_{i=1}^l \mathbb{P}(E_m^i)\leq l(1-c)^m,\] which implies that \[\lim_{m\to \infty} \mathbb{P}(d(\Cech(P^m),O(M))> \epsilon) \leq \lim_{m\to \infty} l(1-c)^m=0.\qedhere\]
\end{proof}

\section{Homotopy Commutative and Homotopy Coherent Interleavings}\label{Sec:Incoherent_Interleavings_And_Rectification}

A simpler candidate definition of the homotopy interleaving distance
can be formulated in terms of \emph{homotopy commutative} interleaving diagrams,
i.e., interleaving diagrams taking values in $\Ho(\Top)$, the homotopy
category of spaces.  In this section, we explore this 
definition and explain why we expect that it is not equal to $d_{HI}$, hence not universal.\footnote{As noted in \cref{Sec:Non-Universality}, following the release of the first version of this paper, Lanari and Scoccola have \cite{lanari2020rectification} have shown that the two distances are indeed unequal.  Aside from some minor corrections and light polish, we have left the text of this section as it was in the first version, and will not mention the results of \cite{lanari2020rectification} again.}

In homotopy theory, one typically avoids working with homotopy commutative diagrams, and works instead with
 richer objects called \emph{homotopy coherent diagrams}, which are homotopically better behaved.  At the end of this section, we observe that $d_{HI}$ admits an equivalent definition in terms of homotopy coherent interleaving diagrams.  

\subsection{Homotopy Commutative Interleavings}

As in \cref{Sec:Model_Cats}, let $\Pi \colon \Top \to
\Ho(\Top)$ denote the functor that takes a space to its
representative in the homotopy category (i.e., the localization with respect to the standard
weak equivalences in $\Top$).

\begin{definition}
A \emph{homotopy commutative $\delta$-interleaving} between $\R$-spaces $X$ and $Y$ is a $\delta$-interleaving in $\Ho(\Top)$ between 
  $\Pi X$ and $\Pi Y$.
 \end{definition}
This definition induces a definition of an interleaving
distance $d_{HC}$ on $\R$-spaces, in the usual way.  The distance $d_{HC}$ is stable, homotopy invariant, and homology bounding.  

\begin{remark}
For many choices of small category $\I$, the natural functor $\Ho(\Top^\I)\to (\Ho(\Top))^\I$ discards some higher order homotopy theoretic information, and this can make it difficult to work directly in the category $(\Ho(\Top))^\I$.  For example, whereas we define 
$\hocolim$ as a functor $\Ho(\Top^\I)\to \Ho(\Top)$, $\Ho(\Top)$
is not cocomplete (it does not even have all pushouts), and so
homotopy colimits generally cannot be defined as functors out of $(\Ho(\Top))^\I$.
\end{remark}

In view of the above remark, one does not expect homotopy commutative $\delta$-interleavings to be homotopically well-behaved objects.  Thus, the definition of $d_{HC}$, while especially simple, is somewhat unnatural.  Nevertheless, one might wonder about the relationship between $d_{HC}$ and $d_{HI}$.  By the universality of $d_{HI}$, we have that $d_{HC}\leq
d_{HI}$.  We expect that $d_{HC}\ne d_{HI}$.  Though we do not have
a proof of this, we will present an example which shows that in the category of based topological spaces, 
homotopy commutative interleavings needn't lift to homotopy interleavings; we
imagine that a similar example can be found which shows that
$d_{HC}<d_{HI}$.   

\subsection{Rectification of Homotopy Commutative Diagrams}
In the next two subsections, it will be convenient for us to work with the category $\Top_{*}$ of \emph{based} CGWH topological spaces and its associated homotopy category $\Ho(\Top_{*})$ \cite[Remark 3.10]{dwyer1995homotopy}.

For $X\colon \I\to \Ho(\Top_*)$, a
\emph{rectification of $X$} is a functor $\tilde X\colon \I\to \Top_*$ such
that $\Pi \tilde X\cong X$.  Rectifications do not always exist.  The
following folklore example, brought to our attention by Tyler Lawson, demonstrates this: 

\begin{example}\label{Ex:No_Rectification_Cube}
Consider the sequence of based maps
\begin{equation}\label{Eq:Rectification}
S^4\xrightarrow{f}S^4\xrightarrow{g}S^3\xrightarrow{h}S^3,
\end{equation}
where $f$ and $h$ are degree 2 maps, and $g$ is the suspension of the Hopf map.  The maps $g\comp f$ and $h\comp g$ are null homotopic.
 
Let $\SimpCat{1}$ denote the category with object set $\{0,1\}$ and a single non-identity morphism $0\to 1$.  We extend the sequence (\ref{Eq:Rectification}) above to a diagram indexed by the cube $\SimpCat{1}^3$ which commutes up to homotopy, as follows:
 \begin{equation}\label{Eq:No_Rect_Cube}
\begin{tikzcd}
& {*} \ar[rr] && S^3 \\
{*} \ar[rr]\ar[ur] &&
{*} \ar[ur] & \\
& {*} \ar[rr] \ar[uu] && S^3 \ar[uu,swap,"h"] 
\\ S^4 \ar[ru]\ar[uu]\ar[rr,swap,"f"] &&
S^4 \ar[uu] \ar[ru,swap,"g"] &
\end{tikzcd}
\end{equation}
A homotopy commutative diagram of this form can be rectified if and only if the Toda bracket $\langle f,g,h \rangle$ contains the trivial map \cite{82516}.  In this case, the Toda bracket consists of the non-zero   element of $\pi_5(S^3)\cong \Z/2\Z$ \cite{toda1962composition}, so the diagram cannot be rectified.
\end{example}

\subsection{Rectification of Interleavings}

To establish that $d_{HC}=d_{HI}$ for based spaces, it would suffice to show that for
any $X,Y\colon \R\to \Top_*$ and homotopy commutative $\delta$-interleaving
$W\colon\I^\delta\to \Ho(\Top_*)$ between $X$ and $Y$, there exists a
rectification $\tilde W\colon\I^\delta\to \Top_*$ of $W$ such that
$\tilde W\circ E^1\htp X$ and $\tilde Z\circ E^2\htp Y$.  However, the next
example shows that such a rectification does not always exist, even if we ignore the conditions on the restrictions
$\tilde W\circ E^i$.

\begin{example}\label{Ex:Unrectifiable}
We can define homotopy commutative interleavings between functors $\ZCat\to \Top_*$ in the same way as for $\R$-spaces.  Leveraging \cref{Ex:No_Rectification_Cube}, we give an example of functors $X,Y\colon \ZCat\to \Top_*$ and a homotopy commutative 2-interleaving between $X$ and $Y$ which cannot be rectified.  It's easy to see that 
that this example extends to yield an unrectifiable homotopy commutative interleaving in the $\RCat$-indexed case, as well.

Letting $f$, $g$, and $h$ be as in \cref{Ex:No_Rectification_Cube}, consider the following homotopy commutative diagram:
\begin{equation}\label{Eq:Interleaving_Form}
\begin{tikzcd}
& & {*} \ar[r] \ar[drr] & S^3 \ar[drr,"h"] & & & \\
S^4 \ar[r,swap,"f"] \ar[urr] & S^4 \ar[r] \ar[urr,swap,"g"] & {*} \ar[r] & {*} \ar[r] & {*} \ar[r] & S^3. \\
\end{tikzcd}
\end{equation}
This diagram clearly extends to a homotopy commutative $2$-interleaving between a pair of functors $X,Y\colon \ZCat\to \Top_*$, by taking the remaining spaces in $X$ and $Y$ to be points.  To check that this cannot be rectified, it suffices to check that \eqref{Eq:Interleaving_Form} cannot be rectified.  To do so, we will observe that \eqref{Eq:Interleaving_Form} can be rectified only if  \eqref{Eq:No_Rect_Cube} can be rectified.  Since \eqref{Eq:No_Rect_Cube} cannot be rectified, this establishes that \eqref{Eq:Interleaving_Form} cannot be rectified.

Let us draw (\ref{Eq:Interleaving_Form}) in a different way, placing each object at a vertex of the cube:
\begin{equation}\label{Eq:Cube_Form}
\begin{tikzcd}
& {*} \ar[rr] && S^3 \\
{*} \ar[rr] &&
{*} \ar[ul] & \\
& {*} \ar[rr] \ar[uu] && S^3 \ar[uu,swap,"h"] 
\\ S^4 \ar[ru] \ar[rr,swap,"f"] &&
S^4.  \ar[uull] \ar[ru,swap,"g"] &
\end{tikzcd}
\end{equation}
Noting that by composition, a commutative diagram of the form
\[
\begin{tikzcd}
a \ar[r] & b \\
c \ar[r] & d \ar[ul] \\
\end{tikzcd}
\]
determines one of the form
\[
\begin{tikzcd}
a \ar[r] & b\\
c \ar[r] \ar[u] & d, \ar[u] \\
\end{tikzcd}
\]
it's clear that if a rectification of \eqref{Eq:Cube_Form} were to
exist, it would yield a rectification of  \eqref{Eq:No_Rect_Cube}.  
Thus \eqref{Eq:Interleaving_Form} cannot be rectified, as we wanted to
show.

Note however that there does exist a homotopy $1$-interleaving between
$X$ and $Y$, obtained by taking all maps between spaces in $X$ and $Y$
to be trivial.  Thus, $d_{HI}(X,Y)=d_{HC}(X,Y)=1$.  Consequently, this example does not establish that $d_{HC}
\neq d_{HI}$ for based spaces. 
\end{example}

\subsection{Homotopy Coherent Interleavings}\label{Sec:Hoco}

We close this section by giving an alternative interpretation of the
homotopy interleaving distance in terms of homotopy coherent diagrams.
As discussed above, a homotopy commutative diagram cannot always be
rectified.  A natural question to ask, then, is what additional
information is required for rectification.  An old theorem of Vogt~\cite{vogt1973homotopy,cordier1986vogt,riehl2018homotopy} provides an answer; it tells us that a functorial rectification exists for \emph{homotopy coherent
diagrams}.

Roughly speaking, a
homotopy coherent diagram is a homotopy commutative diagram together
with explicit choices of all homotopies, homotopies between the
homotopies, and so on.  A formal definition can be given using the language of simplicially enriched functors.  Given a small category $\I$, the homotopy coherent diagrams indexed by $\I$ form a category $\Coh(\I)$ whose morphisms are homotopy classes of homotopy coherent natural transformations.
  
Let $\widetilde\Ho(\Top^I)$ denote the
localization of $\Top^\I$ with respect to objectwise homotopy
equivalences, and recall that $\Ho(\Top^\I)$ denotes the localization of $\Top^\I$ with respect to objectwise weak homotopy
equivalences.   Using Whitehead's theorem, it can be checked that two diagrams in $\Top^\I$ taking values in cofibrant spaces (e.g., CW
complexes) are isomorphic in $\widetilde\Ho(\Top^\I)$ if and only if they are isomorphic in $\Ho(\Top^\I)$.

Vogt's theorem gives an equivalence
\[
\Coh(\I)\to \widetilde \Ho(\Top^\I).
\]
The theorem thus tells us that we can study homotopy coherent diagrams 
using strict diagrams and zig-zags of objectwise homotopy
equivalences.  There are generalizations of Vogt's theorem to diagrams in arbitrary model
categories, e.g., see~\cite[Proposition 4.2.4.4]{LurieHTT}, but we will not need them here.

Guided by these ideas, we formulate a
homotopy-coherent definition of interleavings.

\begin{definition}
A \emph{homotopy coherent
  $\delta$-interleaving} between functors $X,Y\colon \RCat \to \Top$ is a homotopy coherent diagram $Z\in \Coh(\I^\delta)$ such that
$Z\comp E^0\cong X$ and $Z\comp E^1\cong Y$ in $\Coh(\R)$.
\end{definition}

Using basic properties of the the equivalence $\Coh(\I^{\delta})\to \Ho(\Top^{\I^\delta})$
provided by Vogt's theorem, it is
straightforward to prove the following comparison.

\begin{proposition}\label{Prop:d_HI_Characterization}
If there exists a homotopy coherent $\delta$-interleaving between
$\R$-spaces $X$ and $Y$, then there exists a $\delta$-homotopy-interleaving
between $X$ and $Y$.  If $X$ and $Y$ are objectwise cofibrant, then the converse is also true.
\end{proposition}

For objectwise cofibrant $\R$-spaces,
\cref{Prop:d_HI_Characterization} yields a characterization of $d_{HI}$ in
terms of homotopy coherent interleavings.

\section{Towards a Persistent Whitehead Theorem}\label{Sec:Persistent_Whitehead}
To conclude the paper, we explore of the problem of 
formulating a persistent Whitehead theorem.  

\subsection{Persistent Homotopy Groups}
We first need to define persistent homotopy groups.  In the setting of based spaces, this is straightforward.

\begin{definition}
For a functor $X\colon \RCat\to \Top_*$ and $i\geq 0$, we call the composite functor $\pi_i X$ the \emph{$i^{\rm{th}}$ based persistent homotopy group} of $X$.  For $i > 0$, this is a functor $\RCat\to \Grp$, while $\pi_0$ takes values in $\Set$.  
\end{definition}

A related definition of persistent homotopy group appears in \cite{letscher2012persistent}, but concerns only pairs of indices in $\R$.  We have seen that interleavings are defined on objects of $\C^\RCat$, for arbitrary categories $\C$.  Thus, interleavings can be defined in the usual way for based persistent homotopy groups.

\begin{remark}[Unbased Analogues of Persistent Homotopy Groups]
Let us say $X\colon \RCat\to \Top_*$ is \emph{(path) connected} if for each $r\in \R$, $X_r$ (path) connected.  For $X$ path connected, our definition of $\pi_i X$ is reasonable.  However, $\R$-spaces arising in TDA are rarely connected.  For $X$ not connected, the isomorphism type of $\pi_i X$ depends on the choice of basepoints in $X$, and may miss important topological information in components of the spaces $X_r$ not containing the basepoint.  Thus, we wish to define an unbased version of the persistent homotopy group, which keeps track of information at all components.  We also wish to give a definition of interleavings between these objects.   

One way to proceed is to use the definition of a \emph{local system}, as given in \cite[Section 4]{isaksen2001model}.  This definition is functorial, so one can associate a \emph{persistent local system} $\Pi_i X$ to an $\R$-space $X$ for each $i\geq 0$.  Moreover, one has a natural notion of equivalence of local systems, and using this, one can give a definition of interleavings between persistent local systems similar to our definition of the homotopy interleaving distance.  Our triangle inequality argument for the homotopy interleaving distance adapts to give the triangle inequality for this interleaving distance.

That said, a careful study of unbased persistent homotopy groups and their interleavings is beyond the scope of this work; below we restrict attention to connected $\R$-spaces and work with based persistent homotopy groups.
\end{remark}

\subsection{Persistent Whitehead Conjectures}

In a model category $\C$, for cofibrant-fibrant objects $X,Y\in \C$ there is a well-behaved abstract notion of homotopy of morphisms $f,g\colon X\to Y$, generalizing the usual definition of homotopy for maps of topological spaces; see for example \cite[Section 4]{dwyer1995homotopy}.  As in the case of spaces, we write $f\htp g$.  This in turn yields a definition of homotopy equivalence.  The axioms of a model category imply an abstract Whitehead theorem~\cite[Proposition 1.2.8]{hovey}:

\begin{theorem}[Whitehead theorem for model categories]\label{Thm:Abstract_Whitehead}
For any model category $\C$, a weak equivalence between cofibrant-fibrant objects in $\C$ is a homotopy equivalence.
\end{theorem}

The situation in $\Top$ is somewhat simpler: We have a well-behaved homotopy relation (the familiar one) for maps between arbitrary spaces
$X$ and $Y$.  Moreover, in the standard model structure, a cofibrant
object is a retract of a CW-complex 
and all objects are fibrant.  Thus in
this setting, \cref{Thm:Abstract_Whitehead} recovers the classical
result that a weak equivalence between CW-complexes is a homotopy
equivalence.

These good properties carry over to the category of $\R$-spaces.  We have a canonical 
homotopy relation for morphisms between arbitrary $\R$-spaces, which is an equivalence relation; namely, morphisms of $\R$-spaces $f,g:X\to Y$ are homotopic if and only if there exists a morphism $X\times I\to Y$ which restricts to $f$ and $g$, respectively, on $X\times \{0\}$ and $X\times \{1\}$.  Moreover, all $\R$-spaces are fibrant in the projective model structure.

In light of this, we can introduce a persistent generalization of homotopy equivalence for arbitrary $\R$-spaces.  First, note that for $X$ an $\R$-space and $\delta\geq 0$, the
internal maps $\{X_{r,r+\delta}\}_{r\in \R}$ assemble into a morphism
$\varphi^{X,\delta}\colon X\to X(\delta)$.

\begin{definition}
Given $\R$-spaces $X$ and $Y$, we will say a pair of morphisms
$f\colon X\to Y(\delta)$ and $g\colon Y\to X(\delta)$ are
\emph{(inverse) $\delta$-homotopy equivalences} if \[g(\delta) \comp f
\htp \varphi^{X,2\delta}\quad \text{ and }\quad f(\delta) \comp g \htp
\varphi^{Y,2\delta},\] where $f(\delta)\colon X(\delta)\to Y(2\delta)$
is the map induced by $f$, and is $g(\delta)$ defined analogously.
\end{definition}

\begin{remark}\label{Remark:Interleaving_Vs_Approximate_Homotopy}
It can be checked that if cofibrant $\R$-spaces $X$ and $Y$ are $\delta$-homotopy-interleaved, then $X$ and $Y$ are $\delta$-homotopy equivalent.  However, in view of homotopy coherence and rectification issues similar to those discussed in \cref{Sec:Incoherent_Interleavings_And_Rectification}, it is not clear to us whether the converse is true.  
\end{remark}

It is natural to wonder whether for $\R$-spaces, the Whitehead theorem extends to a persistent version as follows:

\begin{naive}\label{NaiveA}
For $X$ and $Y$ connected cofibrant $\R$-spaces, $\delta\geq 0$, and morphism $f\colon X\to Y(\delta)$ with $\pi_i f\colon \pi_iX\to \pi_iY(\delta)$ a $\delta$-interleaving morphism for all $i$, $f$ is a  $\delta$-homotopy equivalence.
\end{naive}

In view of \cref{Remark:Interleaving_Vs_Approximate_Homotopy} and the universality of the homotopy interleaving distance, one might also wonder whether the following is true:

\begin{naive}\label{NaiveB}
Given $X$, $Y$ and $f$ as in the previous conjecture, $X$ and $Y$ are $\delta$-homotopy-interleaved.
\end{naive}

However, the following example makes clear that both conjectures are far from true.  Let $\CWCat$ denote the category of CW-complexes and continuous maps.

\begin{example}\label{Ex:Whitehead}
We specify $X\colon \R\to \CWCat$, $Y\colon \R\to \CWCat$, and $f\colon X\to Y(1)$ satisfying the hypotheses of the above conjectures for $\delta=1$, with $X$ and $Y$ not $\epsilon$-homotopy equivalent, and hence also not $\epsilon$-homotopy-interleaved, for any $\epsilon$.

As a first step towards defining $Y$, for each $n\in \{1,2,\ldots \}$, we define a functor $Y^n\colon \RCat\to \CWCat$ as follows: 
\[Y^n_r:=
\begin{cases}
\underbrace{S^{2^i}\times S^{2^i}\times \cdots \times S^{2^i}}_{2^{n-i} \text{ copies}}&\textup{ for  }r\in [2i,2i+2),\ i\in \{0,1,\ldots n\}\\
*&\textup{ for  }r\in (-\infty,0)\cup [2n+2,\infty),\\
\end{cases}
\]
For $i\geq 0$, we have a map \[S^{2^i}\times S^{2^i}\to S^{2^{i+1}}=S^{2^i}\wedge S^{2^i},\] given by collapsing $S^{2^i}\vee S^{2^i}\subset S^{2^i}\times S^{2^i}$ to a point; here $\vee$ and $\wedge$ denote the wedge product and smash product, respectively.  For $i \in \{0,1,\ldots n-1\}$, $r\in [2i,2i+2)$, and $s\in [2i+2,2i+4)$, we take the internal map $Y^n_{r,s}$ to be the product of $2^{n-i-1}$ copies of this map.  The remaining internal maps in $Y^n$ are specified by composition.  

For example, regarding the torus $S^1\times S^1$ as a quotient of a square in the usual way, the map \[Y^1_{0,2}\colon S^1\times S^1\to S^2\] is the one induced by sending the whole boundary of the square  to a single point, and the map \[Y^2_{0,2}\colon S^1\times S^1\times S^1\times S^1 \to S^2\times S^2\] is equal to $Y^1_{0,2}\times Y^1_{0,2}$.

For all $i$, the map $S^{2^i}\vee S^{2^i}\hookrightarrow S^{2^i}\times S^{2^i}$ induces a surjection on all homotopy groups.  Thus, $\pi_iY^n_{r,r+2}$ is trivial for all $r\in \R$ and $i\geq 0$.  Defining $X':\R\to \CWCat$ by $X'_r=*$ for all $r$, it follows that the trivial morphisms $X'\to Y^n(1)$ and $Y^n\to X'(1)$ induce $1$-interleavings on all based persistent homotopy groups.
 
However, $X'$ and $Y^n$ are not $\delta$-homotopy equivalent for any $\delta<n+1$.  To see this, assume that all spheres in the definition of $Y^n$ are given the usual minimal CW-structure, and note that $Y^n_r=S^{2^n}$ for $r\in [2n,2n+2)$.  The map $Y^n_{0,r}$ acts by collapsing the $(2^n-1)$-skeleton of $Y^n_0$ to a point, so it follows from an easy cellular homology computation that $H_{2^n}(Y^n)_{0,r}\ne 0$.  Thus, $H_{2^n}(Y^n)$ and the trivial module $H_{2^n}(X')$ are not $\delta$-interleaved.  It is straightforward to check that a $\delta$-homotopy equivalence between $\R$-spaces $A$ and $B$ induces a $\delta$-interleaving between $H_iA$ and $H_iB$ for all $i$.  Therefore, $X'$ and $Y^n$ are not $\delta$-homotopy equivalent, as claimed.  On the other hand, $X'$ and $Y^n$ are strictly $(n+1)$-interleaved, via trivial morphisms.  

We next construct $Y':\R\to \CWCat$ such that the trivial morphisms $X'\to Y'(1)$ and $Y'\to X'(1)$ induce $1$-interleavings on all based persistent homotopy groups, but $H_iX'$ and $H_iY'$ are not $\epsilon$-interleaved for any finite $\epsilon$.  To do so, we simply patch together the non-trivial portions of each $Y^n$, taking each morphism between spaces from two different $Y^n$ to be trivial; that is, we take $Y'_r:=Y^1_r$ for $r\in (-\infty,4)$, $Y'_r:=Y^2_{r-4}$ for $r\in [4,10)$, and so on.

Finally, we take $X$ and $Y$ to be cofibrant replacements of $X'$ and $Y'$, respectively, and let $f:X\to Y(1)$ be the cofibrant replacement of the trivial map $X'\to Y'(1)$.  
\end{example}

\cref{Ex:Whitehead} motivates the following weaker pair of persistent Whitehead conjectures:

\begin{conjecture}[Persistent Whitehead Conjectures]\label{Conj:Whitehead}
Suppose we are given connected cofibrant $\R$-spaces $X,Y\colon \RCat\to \CWCat$ with each $X_r$ and $Y_r$ of dimension at most $d$, and $f\colon X\to Y(\delta)$ with $\pi_i f\colon \pi_iX\to \pi_iY(\delta)$ a $\delta$-interleaving morphism for all $i$.  Then there is a constant $c\geq 1$, depending only on $d$, such that 
\begin{enumerate}[(i)]
\item the map $X\to Y(c\delta)$ induced by $f$ is a $c\delta$-homotopy equivalence,
\item $X$ and $Y$ are $c\delta$-homotopy-interleaved.
\end{enumerate}
\end{conjecture}

\bibliographystyle{abbrv}
\bibliography{UH_Refs}

\end{document}